\documentclass[oneside,reqno,a4paper,11pt]{amsart}
\usepackage{amsmath,amsthm,amssymb,amscd,soul,bropd}
\usepackage[scaled=.98,sups,osf]{XCharter}
\usepackage[scaled=1.04,varqu,varl]{inconsolata}
\usepackage[type1]{cabin}
\usepackage[xcharter,varg,varvw,smallerops,subscriptcorrection]{newtxmath}
\usepackage{color,verbatim}
\usepackage{soul}
\soulregister{\eqref}{7}
\usepackage{graphicx}
\usepackage{hyperref}
\usepackage{mathrsfs}
\usepackage{cite}
\usepackage{color}
\usepackage{ifpdf}
\usepackage{fancyhdr}
\usepackage{multirow}
\usepackage{array}
\usepackage{appendix}
\usepackage{longtable}

\newtheorem{thm}{Theorem}[section]
\newtheorem*{thm 1.1}{\textup{\textbf{Theorem 1.1}}}
\newtheorem*{thm 1.2}{\textup{\textbf{Theorem 1.2}}}
\newtheorem*{thm 1.3}{\textup{\textbf{Theorem 1.3}}}

\newtheorem*{assumptionA}{\textup{\textbf{Assumption A}}}
\newtheorem*{assumptionB}{\textup{\textbf{Assumption B}}}
\newtheorem{cor}[thm]{Corollary}
\newtheorem{prop}[thm]{Proposition}
\newtheorem{define}[thm]{Definition}
\newtheorem{rem}[thm]{Remark}

\newtheorem{lemma}[thm]{Lemma}

\newcommand{\p}{\partial}
\newcommand{\beq}{\begin{equation}}
	\newcommand{\eeq}{\end{equation}}
\newcommand{\ben}{\begin{eqnarray}}
	\newcommand{\een}{\end{eqnarray}}
\newcommand{\beno}{\begin{eqnarray*}}
	\newcommand{\eeno}{\end{eqnarray*}}

\numberwithin{equation}{section}
\subjclass[]{}

\begin{document}
	\title{singular profile of free boundary of incompressible inviscid fluid with external force}
	
	\author{LILI DU$^{1,2}$}
	\author{YANG PU$^3$}
	\author{JING YANG$^1$}
	\thanks{* Du is supported by National Nature Science Foundation of China Grant 12125102, Nature Science Foundation of Guangdong Province under Grant 2024A1515012794, and Shenzhen Science and Technology Program JCYJ20241202124209011. Pu is supported by National Nature Science Foundation of China Grant 12201504 and 12371210.}
	\thanks{E-mail: dulili@scu.edu.cn.
	 \quad E-mail: puyang1011@126.com. \quad E-mail: yangj0058@126.com}
	
	\maketitle 
	
	\begin{center}
		$^1$ Department of Mathematics, Sichuan University,
		
		Chengdu 610064, P. R. China.
	\end{center}
	
	\begin{center}
		$^2$	College of Mathematical Sciences, Shenzhen University,
		
		Shenzhen 518061, P. R. China.
	\end{center}
	
	\begin{center}
$^3$School of Mathematics, Big Data Laboratory on Financial Security and Behavior (Laboratory of Philosophy and Social Sciences, Ministry of Education), Southwestern University of Finance and Economics, Chengdu 611130, P. R. China.
	\end{center}
	
	\begin{abstract}
		This article is devoted to investigate the singular profile of the free boundary of two-dimensional incompressible inviscid fluid with external force near the stagnation point. More precisely, given an external force with some polynomial type decay close to the stagnation point, the singular profile of the free boundary at stagnation point possible are corner wave, flat and cusp singularity. Through excluding the cusp and flat singularity, we know the only singular profile is corner wave singularity, and the corner depends on the decay rate of the solution near the stagnation point. The analysis depends on the geometric method to a class of Bernoulli-type free boundary problem with given degenerate gradient function on free boundary. This work is motivated by the significant work [E. V$\breve{a}$rv$\breve{a}$ruc$\breve{a}$ and G. Weiss, Acta Math, 206, 363-403,  (2011)] on Stokes conjecture to the incompressible inviscid fluid acted on by gravity.  
		
		\noindent{keyword: } Free boundary; Singular profile; Euler flow; External force; Blow-up;		
		
	\end{abstract}
	
	\tableofcontents   
	\section{Background and main results}
	
	The mathematical problem of this paper concerns the motion of the interface separating an invisid, incompressible, irrotational fluid, under the influence of an external force $\textbf{\textit{f}}$, from a region of air in 2-dimensional space. The interface, which exists  between air and fluid is called the free boundary, as illustrated in Fig. 1.
	\begin{figure}[h]
		\includegraphics[width=130mm]{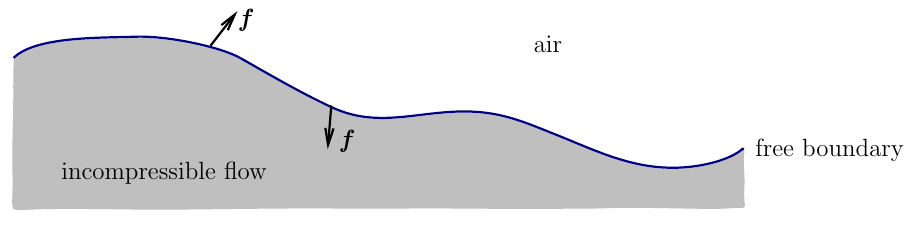}	
		\caption{Physical model}
	\end{figure}
	
	The investigation of flows featuring free boundary between fluid and air has long fascinated both hydrodynamicists and mathematicians, and there is an extensive literature dedicated to the mathematical theory of ideal fluid when influenced by the action of external forces such as gravity, Coriolis force, electromagnetic force, etc. 
	
	The classical problem concerning the motion of water waves in the presence of gravity (e.g., water waves on the surface of the ocean) was first formulated by Isaac Newton in 1687. Since then, there have been some important recent advances, we refer to  
	\cite{IP} and \cite{WS} for two dimensional water waves problem neglecting surface tension.
	
	Additionally, the effect of the earth's solid-body rotation produces Coriolis acceleration, which led to the creation of the Coriolis force named after Coriolis in 1835. Water waves under the influence of Coriolis force are very common in the physics literature, two dimensional obliquely rotating shallow water equations describing a thin inviscid fluid layer flowing over topography in a frame rotating about an arbitrary axis \cite{DP}, an investigation of shallow water equations currents on the propagation of tsunamis \cite{CA} and two-dimensional Camassa-Holm equation arises in the modeling of the propagation of shallow water waves over a flat bed \cite{GG}.
	
	It is well-known that external forces significantly affect the dynamic motion of incompressible flows. Consequently, a central aim of this paper is to investigate the shape of the free boundary at stagnation points under the polynomial type decay of external forces.
	
	\subsection{Mathematical setting of physical problem}
	\
	\newline\indent
	The steady planar flow of an ideal irrotational fluid, which is governed by the isentropic incompressible Euler equations with external forces
	\begin{equation*}
		\begin{aligned}
			\begin{cases}	\operatorname{div}\textbf{u}=0, \\
				(\textbf{u}\cdot\nabla)\textbf{u}+\nabla p=\textbf{\textit{f}},
			\end{cases}
		\end{aligned}
	\end{equation*}
	where the unknown vector function $ \textbf{u}=(u_1(x,y),u_2(x,y)) $ and scalar function $ p(x,y) $ are the fluid velocity and pressure, respectively, and $\textbf{\textit{f}}$ describes the external force.
	
	Assume the fluid  is a zero-vorticity flow characterized by irrotational condition 
	\begin{equation*}
		\operatorname{curl}{\textbf{u}}=0.
	\end{equation*}
	And consequently, the external force $\textbf{\textit{f}}$ is assumed to be conservative, which means that there exists a function $ F(x,y) $ such that
	$$\textbf{\textit{f}}=\nabla F.$$
	Then, we can integrate the equation of momentum energy conservation to obtain Bernoulli’s law
	$$\dfrac{1}{2}|\textbf{u}|^2+p-F(x,y)=\mathcal{B},$$
	where $\mathcal{B}$ is the so-called Bernoulli's constant.
	
	On the free boundary, the fluid satisfies the following slip boundary condition
	\begin{center}
		$\textbf{u}\cdot \textbf{\textbf{n}}=0$.
	\end{center}
	The equation of mass conservation implies that there exists stream function $ u(x,y) $ such that 
	\begin{center}
		$u_1=\p_y u$ and $u_2=-\p_xu$.
	\end{center}
	Then, it is easy to check out that $ u(x,y) $ is harmonic throughout the flow region in view of irrotational condition, whereas the slip boundary condition implies that the free boundary is a level set of $ u(x,y) $. Without loss of generality, one assume that $ \p\{u>0\} $ and $ \{u>0\} $ are the free boundary of the fluid and the fluid region, respectively. Neglecting surface tension, the constant pressure condition holds on the free boundary
	$$p=p_{atm},$$
	where $ p_{atm }$ is the given constant atmopheric pressure.
	
	The Bernoulli’s law gives the gradient condition on the free boundary
	\begin{center}
		$| \nabla u|^2=h(x,y):=2 \left(\mathcal{B}-p_{atm} +F(x,y)\right) \,\,\,\,$ on $\p\{u>0\}$.
	\end{center}
	Since our aim is to get the local singularities, we focus on bounded domain $ \Omega $ which has a non-empty intersection with the free boundary. Along with the harmonic equation, we formulate the following local Bernoulli-type free boundary problem
	\begin{equation}
		\label{eq40}
		\begin{aligned}
			\begin{cases}
				\Delta u =0 &\text{in}\, \Omega\cap\{ u>0\},\\
				\left |\nabla u \right|^2=h(x,y)\,\,&\text{on}\,\Omega\cap\p\{u >0\}.
			\end{cases} 	
		\end{aligned}
	\end{equation}
	\begin{rem}
		It must be noted that the Bernoulli-type free boundary problem above can also describe a traveling wave moving at constant speed on the surface of an incompressible, inviscid, irrotational fluid.
	\end{rem} 
	
	\subsection{Bernoulli-type free boundary problem}
	\
	\newline\indent
	The Bernoulli-type free boundary problem is one of important unconstrained free boundary problems, which arises naturally in a number of physical phenomena. We would like to refer readers to the famous survey \cite{FS} by A. Figalli and H. Shahgholian on it to readers. The solution of Bernoulli-type free boundary problem solves the Laplace equation in its support set and satisfies not only the homogeneous Dirichlet boundary condition, but also the Neumann boundary condition on the free boundary. 
	
	To investigate the properties of  free boundary, the Bernoulli-type problem \eqref{eq40} usually can be divided into two cases according to whether the gradient function $h(x,y)$  degenerates near the free boundary point. The first case is that the gradient function $h(x,y)$ is non-degenerate, namely, there exists a positive constant $c_0$, such that $h(x,y)\geq c_0>0$ on the free boundary. The other case is that there exists a free boundary point $(x^0,y^0)$, such that $h(x^0,y^0)=0$. The point $(x^0,y^0)$ is called the degenerate point of free boundary. In hydrodynamics, a point on the streamline where the velocity is zero is called a stagnation point.
	
	The research on the regularity close to the non-degenerate free boundary point can be dated back to the pioneering work \cite{Alt1981} by H. Alt and L. Caffarelli  in 1981, in which they have showed that there is no singular point on the free boundary in two dimensions, that is to say,  free boundary is locally in $\Omega$ a $C^{1,\alpha}$-surface. However, for higher dimensions ($n\geq3$), G. Weiss in \cite{W} first realized that there existed a maximal dimension $d^*$, such that the singularity of the free boundary could occur only when $n\geq d^*$. Furthermore, it was asserted that $d^*\geq4$ in \cite{CJK}, and $d^*\geq5$ in \cite{JS}. Nevertheless, the regularity theories are widely applied to the physical model with free surface, such as incompressible inviscid jet in \cite{Alt1982,A C F1982,Alt1983,C D X2019}, compressible impinging jet in \cite{C D W2018} and incompressible impinging jet in \cite{CD2023} etc.
	
	On the other hand, the singular wave profile will arise close to the degenerate point of the free boundary, even in two dimensions. This is also related to a very important topic in inviscid water wave, the so-called \textit{Stokes conjecture}. In 1880, G. Stokes conjectured that the free surface of the two-dimensional inviscid irrotational water wave had a symmetric corner of 120 degree at any stagnation point. It has been shown firstly under some structural assumptions by C. Amick, L. Fraenkel and J. Toland \cite{Amick} in 1982 and P. Plotnikov  \cite{Plotnikov2004} in 2002 independently  due to the so-called Nekrasov integral equation. An important breakthrough in the analysis of singular profile is due to the brilliant work \cite{VARVAR2011} by E. V$\breve{a}$rv$\breve{a}$ruc$\breve{a}$ and G. Weiss. They introduced creatively a geometric method to give a rigorous and visual proof to the Stokes conjecture, and all the structural assumptions on the free boundary in the previous work were removed. The proof was based on the blow-up analysis, the monotonicity formula and the frequency formula. Furthermore, the geometric method has also been corroborated to be powerful to deal with the 2D rotational flow \cite{VARVAR2012},
	and the axisymmetric flow in gravitational field \cite{VARVAR2014,DHP,DY}. 
	It is noteworthy that the authors are exploring weak solutions of free boundary problem in \cite{DHP,DY,VARVAR2011,VARVAR2012,VARVAR2014}. Recently, drawing inspiration from \cite{NA}, S. McCurdy considered an interesting problem on the  local minimizer of Alt-Caffarelli functional in \cite{MCC1}, with a degenerate distance function $Q(x)=\operatorname{dist}(x,\Gamma)^\lambda$ for $\lambda>0$ and $\Gamma$ a submanifold. The main content of \cite{MCC1} is using Naber–Valtorta technique to prove the density \textquotedbl gap\textquotedbl \, and containment result such that for one-phase Bernoulli free boundary problems, regularity of the free boundary to the degenerate problem has established.
	Subsequently, in \cite{MCC2}, the author has demonstrated when $\Gamma$ is a flat k-dimensional submanifold, and $\{u>0\}\cap\Gamma=\varnothing$, then the cusp set does not exist.

	The mathematical problem in \cite{VARVAR2011} to attack the Stokes conjecture is formulated into the following two-dimensional one-phase Bernoulli-type free boundary problem
	\begin{equation*}
		\begin{aligned}
			\begin{cases}
				\Delta u =0 &\text{in}\, \Omega\cap\{ u>0\},\\
				\left |\nabla u \right|=\sqrt{2\left( \mathcal{B}-p_{atm}-gy\right) }\,\,&\text{on}\,\Omega\cap\p\{u >0\},
			\end{cases} 	
		\end{aligned}
	\end{equation*}
	where $ g $ is the gravity constant. The force of gravity acting vertically downward is the unique external force considered in \cite{VARVAR2011}. The one key point in the analysis in \cite{VARVAR2011} is that the gradient condition on the free boundary implies the behavior of solution near the stagnation point, that is, $u$ goes like  $ r^{3/2}$ near the stagnation point.
	
	The main motivation of this paper is to consider the singularity of the free boundary to the Euler flow with general external force and meanwhile to analyze the one-phase Bernoulli-type free boundary problem \eqref{eq40} with general gradient function $h(x,y)$. To the end of this paper, we assume the gradient function
	\begin{equation}
		\label{eq41}
		h(x,y)=C|x^{0}-x|^{\alpha}|y^{0}-y|^{\beta},\,\,\,\,\text{ for }\alpha\geq0,\,\beta\geq0,\alpha+\beta>0\,\text{and some constant}\,C,
	\end{equation}
	the set of stagnation points is $\{X=(x,y)|X\in\Omega\cap\p\{u>0\}\cap(\{x=x^0\}\cup\{y=y^0\})\}$.	The problem \eqref{eq40} becomes as follows
	\begin{equation*}
		\begin{cases}
			\Delta u =0\,\,&\text{in $\Omega\cap\{u>0\}$ },\\
			\left |\nabla u \right |^2=C|x^0-x|^{\alpha}|y^0-y|^{\beta}\,\,&\text{on $\Omega\cap\p\{u>0\}$}.
		\end{cases}
	\end{equation*}
	After some normalization, 
	the above problem can then be described by
	\begin{equation}
		\label{eq1}
		\begin{cases}
			\Delta u =0\,\,&\text{in $\Omega\cap\{u>0\}$ },\\
			\left |\nabla u \right |^2=h_0(x,y):=|x|^{\alpha}|y|^{\beta}\,\,&\text{on $\Omega\cap\p\{u>0\}$}.
		\end{cases}
	\end{equation}
	
	In this paper, we focus on the weak solution to the local Bernoulli-type free boundary problem \eqref{eq1} in $\Omega$ in the distributional sense, and we omit the boundary condition on $\p\Omega$.
	
	Obviously, the energy function formally associated with the problem  \eqref{eq1} is 
	$$J(u) :=\int_{\Omega} \left(\left |\nabla u\right |^{2}+|x|^{\alpha}|y|^{\beta}\chi_{\{u>0\}}\right)dX,$$
	which helps us define the weak solution later.	
	\begin{rem}
		It must be remarked that gradient condition on the free boundary implies that the external force in form is 
		\begin{equation*}
			\textbf{\textit{f}}=\nabla F=\dfrac{1}{2}\nabla h_0(x,y)=\frac{1}{2}\left( \alpha|x|^{\alpha-1}|y|^{\beta}\operatorname{sgn}(x),\beta|x|^{\alpha}|y|^{\beta-1}\operatorname{sgn}(y)\right),
		\end{equation*} 
		when $x\neq 0$ and $y\neq 0$, here the function $\operatorname{sgn}(t)$ denotes the sign function of $t$. 
	\end{rem}

	In what follows, we denote $X_0=(x_0,y_0)\in\Omega\cap\p\{u>0\}$ be the possible stagnation point of the free boundary, and analyze the  singular wave profile of the Bernoulli-type free boundary problem \eqref{eq1}. All the possible stagnation points will fall into one of three following categories (see Fig. 2).
	\begin{figure}[h]
		\includegraphics[width=60mm]{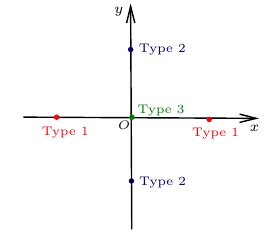}	
		\caption{Three types of stagnation points}
	\end{figure}
	
	\textit{\textbf{Case 1. (Type 1 stagnation point)}  $X_0=(x_0,0)$, for $x_0\neq 0$. Since the $x$ direction is non-degenerate, the singularity arises only in the $ y $ direction, and we suppose that $\beta\geq 1$;}
	
	\textit{\textbf{Case 2. (Type 2  stagnation point)}  $X_0=(0,y_0)$, for $y_0\neq 0$. Since the $y$ direction is non-degenerate, the singularity arises only in the $ x $ direction, and we suppose that $\alpha\geq 1$;}
	
	\textit{\textbf{Case 3. (Type 3 stagnation point)}  $X_0=(0,0)$. The singularity arises both in the $ x $ and $y$ direction, and we suppose that $\alpha\geq1$ and $\beta\geq1$.}
	
\begin{rem}
	The reason we restrict the parameters $\alpha$ and $\beta$ for different cases is to ensure that the external force $\textbf{\textit{f}}$  is a conservative one.
\end{rem}
	
	In addition, $h(x,y)$ can  imply the direction of the external force $\textbf{\textit{f}}$ in Case 1 and 2, briefly speaking, the direction of the force is vertical for Case 1, and the direction of the force is horizontal for Case 2. 
	
	According to the direction of the force and the location of  the stagnation point, we take Case 1 as an example to explain in detail how to classify,
	\begin{equation*}
		\begin{aligned}
			\lim\limits_{x\rightarrow x_0,y\rightarrow 0\pm}\dfrac{\p_y h_0(x,y)}{\p_x h_0(x,y)}=	\lim\limits_{x\rightarrow x_0,y\rightarrow 0\pm}\arctan\dfrac{\beta x}{\alpha y}=\pi\mp\frac{\pi}{2} ,\,\,\,\text{for }\,\,\,x_0<0,
		\end{aligned}
	\end{equation*}
	and 
	\begin{equation*}
		\begin{aligned}
			\lim\limits_{x\rightarrow x_0,y\rightarrow 0\pm}\dfrac{\p_y h_0(x,y)}{\p_x h_0(x,y)}=	\lim\limits_{x\rightarrow x_0,y\rightarrow 0\pm}\arctan\dfrac{\beta x}{\alpha y}=
			\pi\mp	\frac{\pi}{2},\,\,\,\text{for }\,\,\,x_0>0.
		\end{aligned}
	\end{equation*}
	Then we denote the limit above as $ \theta_0 $, which is the deflection angle down from the positive $ x $-axis of $\textbf{\textit{f}}$ at the point $(x_0,0)$.
	Based on the above observation and the direction of the force near the stagnation point, we can divide Case 1 into (see Fig. 3)
	
	\textit{\textbf{Subcase 1.1. (Type 1 stagnation point)}  $x_0< 0$ and $\theta_0=\frac{3\pi}{2}$;}
	
	\textit{\textbf{Subcase 1.2. (Type 1 stagnation point)} $x_0> 0$ and $\theta_0=\frac{\pi}{2}$;}
	
	\textit{\textbf{Subcase 1.3. (Type 1 stagnation point)} $x_0< 0$ and $\theta_0=\frac{\pi}{2}$;}
	
	\textit{\textbf{Subcase 1.4. (Type 1 stagnation point)} $x_0> 0$ and $\theta_0=\frac{3\pi}{2}$. }
	\begin{figure}[h]
		\includegraphics[width=150mm]{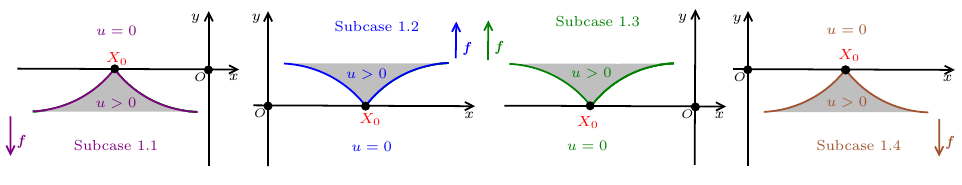}	
		\caption{Type 1 stagnation points }
	\end{figure}
	\begin{rem}
		For the special case (Subcase 1.1 and Subscase 1.4), $\alpha=0$ and $\beta=1$, 	 the free boundary problem \eqref{eq1} is nothing but the water wave problem under gravity which was discussed in \cite{VARVAR2011} by E. V$\breve{a}$rv$\breve{a}$ruc$\breve{a}$ and G. Weiss.
	\end{rem}
	
	It is easy to see that the direction of the force is $ 0 $ or $\pi$ in Case 2, and the stagnation point is $ (0,y_0) $. It follows that Case 2 is divided into the following subcases (see Fig. 4)
	
	\textit{\textbf{Subcase 2.1. (Type 2 stagnation point)} $y_0<0$ and $\theta_0=\pi$;}
	
	\textit{\textbf{Subcase 2.2. (Type 2 stagnation point)} $y_0> 0$ and $\theta_0=0$;}
	
	\textit{\textbf{Subcase 2.3. (Type 2 stagnation point)} $y_0< 0$ and $\theta_0=0$;}
	
	\textit{\textbf{Subcase 2.4. (Type 2 stagnation point)} $y_0> 0$ and $\theta_0=\pi$.}

	\begin{figure}[h]
		\includegraphics[width=150mm]{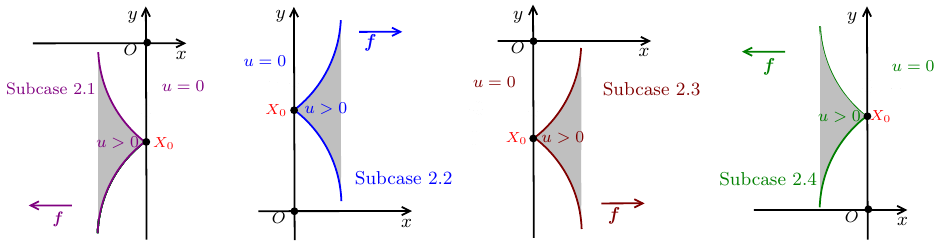}	
		\caption{Type 2 stagnation points }
	\end{figure}

	Roughly speaking, the gradient assumption in \eqref{eq1} gives the deflection angle of the external force at Type 1 and Type 2 stagnation points. However, for Case 3, the gradient condition 	$ \left |\nabla u \right |^2=h_0(x,y)$ implies that the external force $\textbf{\textit{f}}=\nabla h_0(x,y)$ vanishes at the origin, and the deflection angle of the external force is not well defined.
	
	In this paper, we mainly focus on the asymptotic singular profile of the wave near the stagnation point, and denote  $\theta^*$ the angle bisector of the asymptotic directions of the free boundary. Moreover, we conjecture that a reasonable conclusion is that the direction of the force coincides with the angle $ \theta^* $, that is to say $\theta_0=\theta^*$. This conjecture will be confirmed in Case 1 and Case 2. However, since the force vanishes at the stagnation point in Case 3, the direction of the force is missing now. Therefore, we have to give an additional assumption on the angle bisector $\theta^*$ of the asymptotic directions of  free boundary for Case 3.
	
	\textit{\textbf{Case 3. (Type 3 stagnation point)} $X_0=(0,0)$, we denote $\nu^*=(\cos\theta^*,\sin\theta^*)$. See Fig. 5.}
	
	\begin{figure}[h]
		\includegraphics[width=50mm]{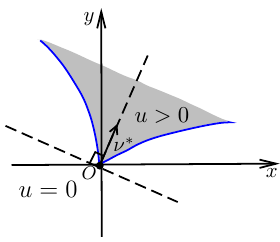}
		\caption{Type 3 stagnation point}
	\end{figure}
\begin{rem}
It should be noted that the first result on the regularity and singularity of the free boundary to the Alt-Caffarelli functional with a general degenerate function $Q(x)$ was established by S. McCurdy in the recent interesting work \cite{MCC1} and it was proved that the free boundary may be decomposed into a rectifiable set and a cusp set by using Naber–Valtorta technique.
However, the results strictly relied on the structure of the degenerate funciton $Q(x)$, which is the geometric structure of the free boundary to the submanifold. In this paper, especially for Type 3 stagnation point, the gradient function  $|x|^{\alpha}|y|^{\beta}$ does not fall into the structural assumption and we have to analyze the blow-up limit and the density carefully.
 
\end{rem}

	\subsection{Definitions and assumptions}
	\
	\newline\indent
	In this subsection, we give the definitions of weak solution, blow-up sequence, blow-up limit and weighted density.
	\begin{define}
		(Stagnation points). We call $$S^u:=\left\lbrace X=(x,y)|X\in\Omega\cap\p\{u>0\}\cap(\{x=0\}\cup\{y=0\})\right\rbrace $$
		as the set of stagnation points.
	\end{define}

	\begin{define}
		(Weak solution) A non-negative function $u\in W_{\operatorname{loc}}^{1,2}(\Omega)$ is defined as a weak solution of \eqref{eq1} in $\Omega$, if 
		
		(i) $u\in C^0(\Omega)\cap  C^2\left(\Omega\cap\{u>0\}\right)$, $u\geq0$ in $\Omega$, $u\equiv0$ in $\{X|(X-X_0)\cdot(\cos\theta_0,\sin\theta_0)\leq0\}$;
		
		(ii) the first variation with respect to domain variations of the functional 
		$$J(v) :=\int_{\Omega} \left(\left |\nabla v\right |^{2}+|x|^{\alpha}|y|^{\beta}\chi_{\{v>0\}}\right)dX$$    
		vanishes at $ v=u $, and $\chi_{\{v>0\}}$ is the characteristic function of a set  ${\{v>0\}}$. Equivalently, 
		\begin{equation}
			\label{eq23}
			\begin{aligned}
				0=&-\frac{d}{d\varepsilon}J\left(u\left( X+\varepsilon \phi(X)\right) \right)\Big|_{\varepsilon =0}\\
				=&\int_{\Omega}\Big[\left |\nabla u\right |^{2}\operatorname{div} \phi-2\nabla uD\phi \nabla u+ |x|^{\alpha}|y|^{\beta}\operatorname{div}\phi\chi_{\{u>0\}}\\
				&+\big( \alpha|x|^{\alpha-1}|y|^{\beta}\phi_1\operatorname{sgn}(x)
				+\beta|x|^{\alpha}|y|^{\beta-1}\phi_2\operatorname{sgn}(y)\big)\chi_{\{u>0\}}\Big]dX 
			\end{aligned}
		\end{equation}
		for any $ \phi=(\phi_1,\phi_2)\in C_{0}^{1}(\Omega;\mathbb{R}^{2}),$ we denote $dX=dxdy$ throughout this paper;
		
		(iii) the topological free boundary $(\p\{u>0\}\cap\Omega)\backslash S^u$ is locally a $C^{2,\gamma}$-surface for some $\gamma\in(0,1)$.
	\end{define}
	\begin{rem}
		For Case 3,   the force degenerates at the origin, which means that the force's direction cannot be determined by the given conditions at this point. Hence we use $\theta^*$ instead of $\theta_0$ in the definition of the weak solution.
	\end{rem}
	\begin{rem}
		Indeed, for the regularity of the topological free boundary, the fact (iii) in Definition 1.7 has been verified in the pioneer work \cite{Alt1981} by H. Alt and L. Caffarelli.
	\end{rem}
	
	We assume that the following Bernstein estimate holds close to the free boundary.
	\begin{assumptionA}
		In regards to the free boundary problem \eqref{eq1}, let $u$ satisfy
		\begin{equation}
			\label{eq28}
			|\nabla u|^2\leq C |x|^{\alpha}|y|^{\beta}
		\end{equation}
		in $B_{r_0}(X_0)$, for $X_0\in S^u$, and some $r_0>0$, $C$ is a positive constant. 
	\end{assumptionA}

	Next, we will provide the definitions of the blow-up sequence for three types of stagnation points respectively. It is noticeable that the gradient condition $|\nabla u|^2=|x|^{\alpha}|y|^{\beta}$ on the free boundary implies the behavior of the weak solution $u$ close to the stagnation point. In other words, $u$ goes like $r^{\frac{\beta+2}{2}}$ near 	Type 1 stagnation point $X_0=(x_0,0)$;  $u$ goes like $r^{\frac{\alpha+2}{2}}$ near 	Type 2 stagnation point $X_0=(0,y_0)$;  $u$ goes like $r^{\frac{\alpha+\beta+2}{2}}$ near 	Type 3 stagnation point $X_0=(0,0)$. For this reason, we need to give the following definitions of rescaled functions.
	\begin{define}
		For every $r>0$, we define the rescaled functions
		\begin{center}
			$u(X)=r^{\kappa_i}u(X_0+rX),\,\,\text{for}\,i=1,2,3,$
		\end{center}
		where $\kappa_1=-\frac{\beta+2}{2}$, $\kappa_2=-\frac{\alpha+2}{2}$ and $\kappa_3=-\frac{\alpha+\beta+2}{2}$ respectively for Type 1, 2, and 3 stagnation point. Let $\{r_m\}$ be a vanishing sequence of positive numbers as $m\rightarrow\infty$, we say that the sequence of functions $u_m(X)=r_m^{\kappa_i}u(X_0+r_mX)$ are blow-up sequences.
		
		In addition, for each $r>0$, if the sequence $\{u_m(X)\}$ is uniformly bounded in $ W^{1,2}(B_r)$, there exists a subsequence converging to $u_0(X)$ weakly in $ W^{1,2}(B_r)$ as $m\rightarrow\infty$. Such a function  $u_0(X)$ is called the blow-up limit of  $u_m(X)$ at the stagnation point $X_0$, which indicates the infinitesimal behavior of $u(X)$ near $X_0$.
	\end{define}
	
	The following definitions are dedicated to the monotonicity formula for the adjusted boundary functional introduced by G. Weiss in \cite{WG}. Precisely, for three types of stagnation points, $u$ goes like $r^{\frac{\beta+2}{2}}$, $r^{\frac{\alpha+2}{2}}$ and $r^{\frac{\alpha+\beta+2}{2}}$ respectively, accordingly we define the Weiss adjusted boundary energies $M_i(r)$ as follows
	\begin{equation}
		\label{eq66}
		\begin{aligned}
			M_i(r)=r^{2\kappa_i}\int_{B_{r}(X_0)}\left(\left |\nabla u\right |^{2}+|x|^{\alpha}|y|^{\beta}\chi_{\{u>0\}}\right)dX+\kappa_i r^{2\kappa_i-1}\int_{\p B_{r}(X_0)}u^{2}d\mathcal{H}^{1},\,\text{for}\,i=1,2,3.
		\end{aligned}
	\end{equation}
	
	\begin{rem}
		In the next section, we will prove that the limits $M_i(r_m)$ exist as $m\rightarrow \infty$ and denote the limits as $M_i(0+)$.
		Taking Subscase 1.1 , 2.1 and Case 3 as examples,
		\begin{equation*}
			M_1(0+)=(-x_0)^{\alpha}\lim\limits_{m \rightarrow\infty}\int_{B_1}(y_-)^{\beta}\chi_{\{u_m>0\}}dX=(-x_0)^{\alpha}\lim\limits_{m \rightarrow\infty}\dfrac{w_2\int_{B_1}(y_-)^{\beta}\chi_{\{u_m>0\}}dX}{|B_1|};
		\end{equation*}
		$$M_2(0+)=(-y_0)^{\beta}\lim\limits_{m    \rightarrow\infty}\int_{B_1}(x_-)^{\alpha}\chi_{\{u_m>0\}}dX=(-y_0)^{\beta}\lim\limits_{m    \rightarrow\infty}\dfrac{w_2\int_{B_1}(x_-)^{\alpha}\chi_{\{u_m>0\}}dX}{|B_1|};$$
		and 
		$$M_3(0+)=\lim\limits_{m    \rightarrow\infty}\int_{B_1}|x|^{\alpha}|y|^{\beta}\chi_{\{u_m>0\}}dX=\lim\limits_{m    \rightarrow\infty}\dfrac{w_2\int_{B_1}|x|^{\alpha}|y|^{\beta}\chi_{\{u_m>0\}}dX}{|B_1|},$$
		where $y_+=\operatorname{max}\{y,0\}$ and  $y_-=-\operatorname{min}\{y,0\}$, $w_2$ is the volume of the unit ball in two dimensions, which are equivalent to the weighted  density of the set $\{u>0\}$ at $X_0$ in Subscase 1.1 , 2.1 and Case 3 respectively. 
		
		The notion of weighted density was first introduced by E. V$\breve{a}$rv$\breve{a}$ruc$\breve{a}$  and G. Weiss in \cite{VARVAR2011}, providing valuable insights into the study of Stokes conjectures. This constitutes a brilliant study, and  as  such, we adopt their definition in our paper.
	\end{rem}
	
	Under the previous assumptions, we are only able to achieve exact forms of the blow-up limit and weighted densities. In order to obtain the asymptotic behavior of free boundary, the following assumption is required.
	\begin{assumptionB}
		For $X_0\in S^u$, let $\p\{u>0\}$ be a continuous injective curve $\sigma=(\sigma_1(t),\sigma_2(t)):I\rightarrow\mathbb{R}^2$ in a neighborhood of $X_0$, such that $\sigma(0)=X_0$, and $\sigma_1(t)-\sigma_1(0)\neq0$ for $(-s,s)\backslash\{0\}$, where $I$ is an interval in $ \mathbb{R}^2 $ containing the origin and $s $ is a positive constant.
	\end{assumptionB}
	
	\subsection{Main results}
	\
	\newline\indent
	In this  paper, we consider the asymptotic singularity profile of the free boundary to the Bernoulli-type free boundary problem \eqref{eq1}, with boundary gradient function 
	$$h_0(x,y)=|x|^{\alpha}|y|^{\beta},$$
	and give different singular profiles as follows, see Table 1 in Appendix.

	\begin{thm 1.1}
		(Type 1 stagnation point). Let $u$ be a weak solution of the free boundary problem \eqref{eq1} satisfying Assumption A and B. Then for $\alpha\geq0$ and $\beta\geq1$, at Type 1 stagnation point, the singular profile must be a $ \frac{2\pi}{\beta+2} $-degree corner wave profile, that means the wave has a symmetric corner of $ \frac{2\pi}{\beta+2} $ at the stagnation point. Moreover,
		
		(1) for Subcase 1.1, the blow-up limit is 
		\begin{center}
			$u_0(X)=u_0(r\cos\theta,r\sin\theta)=\frac{2(-x_0)^{\frac{\alpha}{2}}}{\beta+2}r^\frac{\beta+2}{2}\cos^{\frac{\beta}{2}}\left( \frac{\pi}{\beta+2}\right)  \cos\left( \frac{\beta+2}{2}\theta+\frac{\beta+2}{4}\pi\right) \chi_{\left\lbrace-\frac{\beta+4}{2(\beta+2)}\pi <\theta<-\frac{\beta}{2(\beta+2)}\pi\right\rbrace },$
		\end{center}
		and	the corresponding weighted density is 
		$$M_1(0+)=(-x_0)^{\alpha}\int_{B_1\bigcap\left\lbrace -\frac{\beta+4}{2(\beta+2)}\pi <\theta<-\frac{\beta}{2(\beta+2)}\pi\right\rbrace } (y_-)^{\beta}dX,$$
		and
		\begin{equation}
			\label{eq99}
			\lim\limits_{t\rightarrow0\pm}\frac{\sigma_2(t)}{\sigma_1(t)-x_0}=-\tan \frac{\left( \beta+2\pm2\right) \pi}{2(\beta+2)} \,\,\,\text{or}\lim\limits_{t\rightarrow0\pm}\frac{\sigma_2(t)}{\sigma_1(t)-x_0}=-\tan\frac{\left( \beta+2\mp2\right)\pi }{2(\beta+2)}.
		\end{equation}
		The corner wave profile can be seen in Fig. 6(a).
		
		(2) For Subcase 1.2, the blow-up limit is
		\begin{center}
			$u_0(X)=u_0(r\cos\theta,r\sin\theta)=\frac{2x_0^{\frac{\alpha}{2}}}{\beta+2}r^\frac{\beta+2}{2}\sin^{\frac{\beta}{2}}\left( \frac{\beta}{2(\beta+2)}\pi\right)  \cos\left( \frac{\beta+2}{2}\theta-\frac{\beta+2}{4}\pi\right) \chi_{\left\lbrace\frac{\beta}{2(\beta+2)}\pi<\theta<\frac{\beta+4}{2(\beta+2)}\pi\right\rbrace},$	
		\end{center}
		and the corresponding weighted density is 
		$$M_1(0+)=x_0^{\alpha}\int_{B_1\bigcap\left\lbrace \frac{\beta}{2(\beta+2)}\pi <\theta<\frac{\beta+4}{2(\beta+2)}\pi\right\rbrace } (y_+)^{\beta}dX,$$
		and
		$$\lim\limits_{t\rightarrow0\pm}\frac{\sigma_2(t)}{\sigma_1(t)-x_0}=\tan \frac{\left( \beta+2\mp2\right)\pi}{2(\beta+2)} \,\,\,\text{or}\lim\limits_{t\rightarrow0\pm}\frac{\sigma_2(t)}{\sigma_1(t)-x_0}=\tan\frac{\left( \beta+2\pm2\right)\pi }{2(\beta+2)}.$$
		The corner wave profile can be seen in Fig. 6(b).
		\begin{figure}[!h]
			\includegraphics[width=130mm]{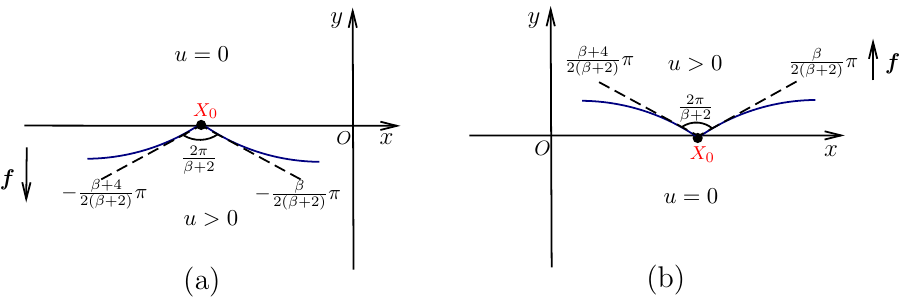}	
			\caption{Singular profile for Subscase 1.1 and Subcase 1.2}
		\end{figure}
		
		(3) For Subcase 1.3, the blow-up limit is  
		\begin{center}
			$u_0(X)=u_0(r\cos\theta,r\sin\theta)=\frac{2(-x_0)^{\frac{\alpha}{2}}}{\beta+2}r^\frac{\beta+2}{2}\sin^{\frac{\beta}{2}}\left( \frac{\beta}{2(\beta+2)}\pi\right)  \cos\left( \frac{\beta+2}{2}\theta-\frac{\beta+2}{4}\pi\right) \chi_{\left\lbrace\frac{\beta}{2(\beta+2)}\pi<\theta< \frac{\beta+4}{2(\beta+2)}\pi\right\rbrace} ,$
		\end{center}
		and the corresponding weighted density is
		$$M_1(0+)=(-x_0)^{\alpha}\int_{B_1\bigcap\left\lbrace \frac{\beta}{2(\beta+2)}\pi <\theta<\frac{\beta+4}{2(\beta+2)}\pi\right\rbrace } (y_+)^{\beta}dX,$$
		and
		$$\lim\limits_{t\rightarrow0\pm}\frac{\sigma_2(t)}{\sigma_1(t)-x_0}=\tan \frac{\left( \beta+2\mp2\right)\pi}{2(\beta+2)} \,\,\,\text{or}\lim\limits_{t\rightarrow0\pm}\frac{\sigma_2(t)}{\sigma_1(t)-x_0}=\tan\frac{\left( \beta+2\pm2\right)\pi }{2(\beta+2)}.$$
		The corner wave profile can be seen in Fig. 7(a).
		
		(4) For Subcase 1.4, the blow-up limit is 
		\begin{center}
			$u_0(X)=u_0(r\cos\theta,r\sin\theta)=\frac{2x_0^{\frac{\alpha}{2}}}{\beta+2}r^\frac{\beta+2}{2}\cos^{\frac{\beta}{2}}\left( \frac{\pi}{\beta+2}\right)  \cos\left( \frac{\beta+2}{2}\theta+\frac{\beta+2}{4}\pi\right) \chi_{\left\lbrace-\frac{\beta+4}{2(\beta+2)}\pi <\theta<-\frac{\beta}{2(\beta+2)}\pi\right\rbrace },$	
		\end{center}
		and the corresponding weighted density is
		$$M_1(0+)=x_0^{\alpha}\int_{B_1\bigcap\left\lbrace -\frac{\beta+4}{2(\beta+2)}\pi <\theta<-\frac{\beta}{2(\beta+2)}\pi\right\rbrace } (y_-)^{\beta}dX,$$
		and
		$$\lim\limits_{t\rightarrow0\pm}\frac{\sigma_2(t)}{\sigma_1(t)-x_0}=-\tan \frac{\left( \beta+2\pm2\right) \pi}{2(\beta+2)} \,\,\,\text{or}\lim\limits_{t\rightarrow0\pm}\frac{\sigma_2(t)}{\sigma_1(t)-x_0}=-\tan\frac{\left( \beta+2\mp2\right)\pi }{2(\beta+2)}.$$
		The corner wave profile can be seen in Fig. 7(b).
		\begin{figure}[!h]
			\includegraphics[width=130mm]{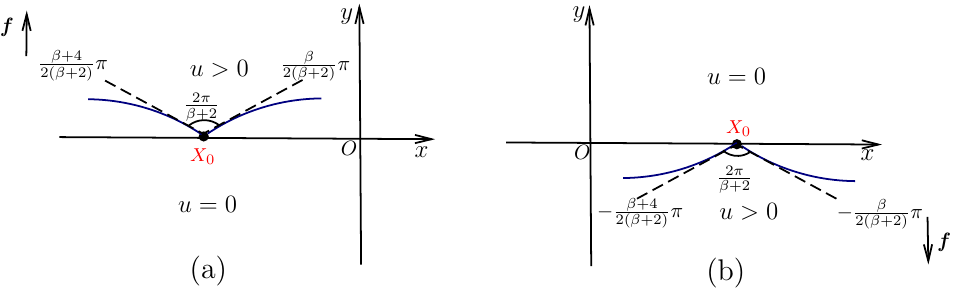}	
			\caption{Singular profile for  Subscase 1.3 and Subcase 1.4}
		\end{figure}
	\end{thm 1.1}
	\begin{rem}
		As mentioned in Theorem 1.1, the value of $\beta$ is closely related to the angle of the asymptotic directions of the free boundary. The angle of the asymptotic directions of the free boundary is $\frac{2\pi}{\beta+2}\in(0,\frac{2\pi}{3}]$. The larger $\beta$ is, the faster the solution decays, thus the angle becomes smaller.
	\end{rem}
	\begin{rem}
		The reason there are two limits for the slope in \eqref{eq99} is that we do not require $t>0$ to indicate $X$ converging to $X_0$ from the positive direction, and at the same time we do not require $t<0$ to indicate $X$ converging to $X_0$ from the negative direction. But the geometrical meaning of them is the same and both represent two asymptotic directions of the free boundary which are unique.
	\end{rem}
	\begin{thm 1.2}
		(Type 2 stagnation point). Let $u$ be a weak solution of the free boundary problem \eqref{eq1}  satisfying Assumption A and B. Then when $\alpha\geq1$ and $\beta\geq0$, at Type 2 stagnation point, the singular profile must be a $ \frac{2\pi}{\alpha+2} $-degree corner wave profile, that means the wave has a symmetric corner of $ \frac{2\pi}{\alpha+2} $ at the stagnation point. Moreover,
		
		(1) for Subcase 2.1, the blow-up limit is 
		\begin{center}
			$u_0(X)=u_0(r\cos\theta,r\sin\theta)=\frac{2(-y_0)^{\frac{\beta}{2}}}{\alpha+2}r^\frac{\alpha+2}{2}\cos^{\frac{\alpha}{2}}\left( \frac{\alpha+1}{\alpha+2}\pi\right) \cos\left( \frac{\alpha+2}{2}\theta+\frac{2-\alpha}{2}\pi\right) \chi_{\left\lbrace \frac{\alpha+1}{\alpha+2}\pi<\theta<\frac{\alpha+3}{\alpha+2}\pi\right\rbrace },$	
		\end{center}
		and the corresponding weighted density is 
		$$M_2(0+)=(-y_0)^{\beta}\int_{B_1\bigcap\left\lbrace\frac{\alpha+1}{\alpha+2}\pi<\theta<\frac{\alpha+3}{\alpha+2}\pi\right\rbrace } (x_-)^{\alpha}dX,$$ 
		and
		$$\lim\limits_{t\rightarrow0\pm}\frac{\sigma_2(t)-y_0}{\sigma_1(t)}=\tan \frac{\left( \alpha+2\pm1\right) \pi}{\alpha+2} \,\,\text{or}\lim\limits_{t\rightarrow0\pm }\frac{\sigma_2(t)-y_0}{\sigma_1(t)}=\tan\frac{\left( \alpha+2\mp1\right) \pi}{\alpha+2}.$$
		The corner wave profile can be seen in Fig. 8(a).
		
		(2) For Subcase 2.2, the blow-up limit is
		\begin{center}
			$u_0(X)=u_0(r\cos\theta,r\sin\theta)=\frac{2y_0^{\frac{\beta}{2}}}{\alpha+2}r^\frac{\alpha+2}{2}\cos^{\frac{\alpha}{2}}\left( \frac{\pi}{\alpha+2}\right) \cos\left( \frac{\alpha+2}{2}\theta\right) \chi_{\left\lbrace -\frac{\pi}{\alpha+2}<\theta<\frac{\pi}{\alpha+2}\right\rbrace },$
		\end{center}
		and the corresponding weighted density is
		$$M_2(0+)=y_0^{\beta}\int_{B_1\bigcap\left\lbrace-\frac{\pi}{\alpha+2}<\theta<\frac{\pi}{\alpha+2}\right\rbrace } (x_+)^{\alpha}dX.,$$ 
		and
		$$\lim\limits_{t\rightarrow0\pm}\frac{\sigma_2(t)-y_0}{\sigma_1(t)}=\pm\tan \frac{ \pi}{\alpha+2} \,\,\text{or}\lim\limits_{t\rightarrow0\pm}\frac{\sigma_2(t)-y_0}{\sigma_1(t)}=\mp\tan\frac{ \pi}{\alpha+2}.$$
		The corner wave profile can be seen in Fig. 8(b).
		\begin{figure}[h]
			\includegraphics[width=130mm]{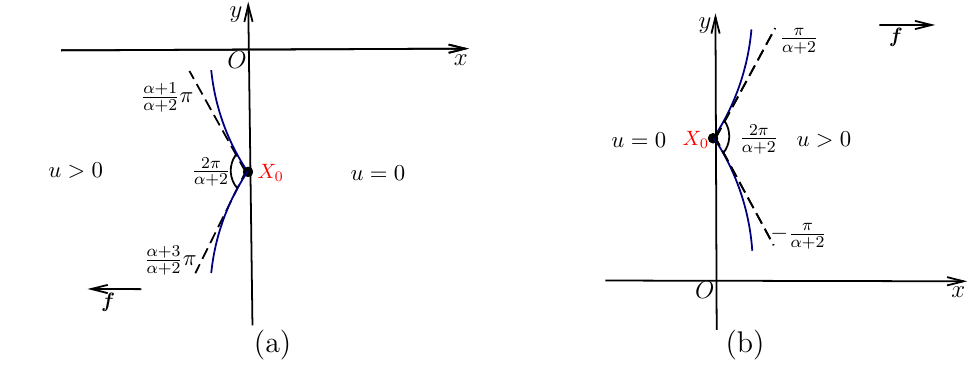}	
			\caption{Singular profile for Subscase 2.1 and Subcase 2.2}
		\end{figure}
		
		(3) For Subcase 2.3, the blow-up limit is
		\begin{center}
			$u_0(X)=u_0(r\cos\theta,r\sin\theta)=\frac{2(-y_0)^{\frac{\beta}{2}}}{\alpha+2}r^\frac{\alpha+2}{2}\cos^{\frac{\alpha}{2}}\left( \frac{\pi}{\alpha+2}\right) \cos\left( \frac{\alpha+2}{2}\theta\right) \chi_{\left\lbrace -\frac{\pi}{\alpha+2}<\theta<\frac{\pi}{\alpha+2}\right\rbrace },	$
		\end{center}
		and the corresponding weighted density is
		$$M_2(0+)=(-y_0)^{\beta}\int_{B_1\bigcap\left\lbrace-\frac{\pi}{\alpha+2}<\theta<\frac{\pi}{\alpha+2}\right\rbrace } (x_+)^{\alpha}dX,$$ 
		and
		$$\lim\limits_{t\rightarrow0\pm}\frac{\sigma_2(t)-y_0}{\sigma_1(t)}=\pm\tan \frac{ \pi}{\alpha+2} \,\,\text{or}\lim\limits_{t\rightarrow0\pm}\frac{\sigma_2(t)-y_0}{\sigma_1(t)}=\mp\tan\frac{ \pi}{\alpha+2}.$$
		The corner wave profile can be seen in Fig. 9(a).
		
		(4) For Subcase 2.4, the blow-up limit is 
		\begin{center}
			$u_0(X)=u_0(r\cos\theta,r\sin\theta)=\frac{2y_0^{\frac{\beta}{2}}}{\alpha+2}r^\frac{\alpha+2}{2}\cos^{\frac{\alpha}{2}}\left( \frac{\alpha+1}{\alpha+2}\pi\right) \cos\left( \frac{\alpha+2}{2}\theta+\frac{2-\alpha}{2}\pi\right) \chi_{\left\lbrace \frac{\alpha+1}{\alpha+2}\pi<\theta<\frac{\alpha+3}{\alpha+2}\pi\right\rbrace },$
		\end{center}
		and the corresponding weighted density is
		$$M_2(0+)=y_0^{\beta}\int_{B_1\bigcap\left\lbrace\frac{\alpha+1}{\alpha+2}\pi<\theta<\frac{\alpha+3}{\alpha+2}\pi\right\rbrace } (x_-)^{\alpha}dX,$$
		and
		$$\lim\limits_{t\rightarrow0\pm}\frac{\sigma_2(t)-y_0}{\sigma_1(t)}=\tan \frac{\left( \alpha+2\pm1\right) \pi}{\alpha+2} \,\,\text{or}\lim\limits_{t\rightarrow0\pm}\frac{\sigma_2(t)-y_0}{\sigma_1(t)}=\tan\frac{\left( \alpha+2\mp1\right) \pi}{\alpha+2}.$$
		The  corner wave profile can be seen in Fig. 9(b).	
		\begin{figure}[h]
			\includegraphics[width=120mm]{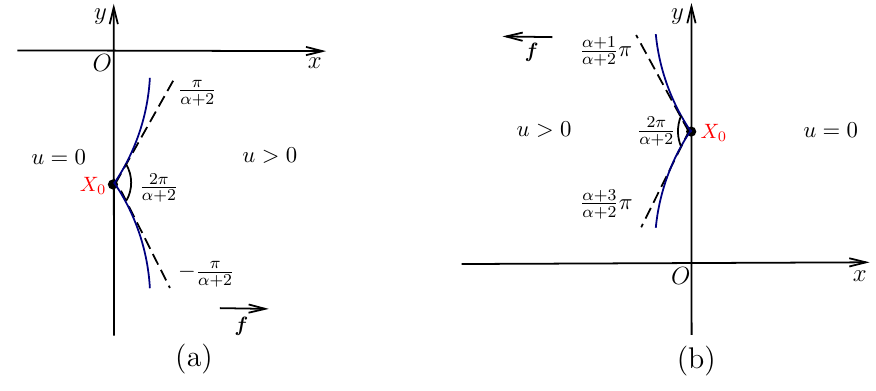}	
			\caption{Singular profile  for Subscase 2.3 and Subscase 2.4}
		\end{figure}		
	\end{thm 1.2}
	\begin{thm 1.3}
		(Type 3 stagnation point).
		Let $u$ be a weak solution of the free boundary problem \eqref{eq1}	satisfying Assumption A and B, then for $\alpha\geq1$ and $\beta\geq1$, at Type 3 stagnation point, the singular profile must be a $ \frac{2\pi}{\alpha+\beta+2} $-degree corner wave profile, that means the wave has a  corner of $ \frac{2\pi}{\alpha+\beta+2} $ at the stagnation point. 
		
		More precisely, at each Type 3 stagnation point, there exist $\theta_1$ and $\theta_2$, such that $\theta_2-\theta_1=\frac{2\pi}{\alpha+\beta+2}$ and satisfy the boundary condition of blow-up limit.	 The blow-up limit is as follows
		\begin{equation*}
			\begin{aligned}
				u_0(X)=&u_0(r\cos\theta,\sin\theta)\\
				=&\dfrac{2|\cos\theta_1|^{\frac{\alpha}{2}}|\sin\theta_1|^{\frac{\beta}{2}}}{\alpha+\beta+2}r^\frac{\alpha+\beta+2}{2}  \cos\left( \frac{\alpha+\beta+2}{2}\theta+\arcsin\left|\cos\left( \frac{\alpha+\beta+2}{2}\theta_1\right)  \right| \right) \chi_{\left\lbrace\theta_1 <\theta<\theta_2\right\rbrace }.
			\end{aligned}
		\end{equation*}
		And the  corresponding density is 	
		$$M_3(0+)=\int_{B_1\cap\left\lbrace \theta_1 <\theta<\theta_2\right\rbrace } |x|^{\alpha}|y|^{\beta}dX,$$
		and 
		$$\lim\limits_{t\rightarrow0\pm}\frac{\sigma_2(t)}{\sigma_1(t)}=\tan\left( \dfrac{\theta_1+\theta_2}{2}\pm\dfrac{\theta_1-\theta_2}{2}\right)  \,\,\text{or}\lim\limits_{t\rightarrow0\pm}\frac{\sigma_2(t)}{\sigma_1(t)}=\tan\left( \dfrac{\theta_1+\theta_2}{2}\pm\dfrac{\theta_2-\theta_1}{2}\right) .$$ 
		The corner wave profile can be seen in Fig. 10. 
		\begin{figure}[h]
			\includegraphics[width=50mm]{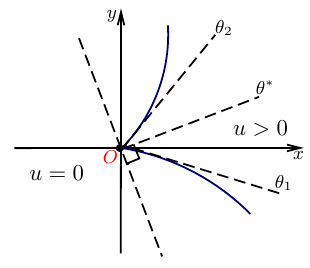}	
			\caption{Singular profile for Type 3 stagnation point }
		\end{figure}
	\end{thm 1.3}
	\begin{rem}
		In Case 3, the external force \textbf{\textit{f}} vanishes at the origin, and then we can not define the direction of the force at the stagnation point. However, due to introducing a potential  function, we can formulate that it solves a Tschebysheff equation. And then it gives the homogeneity of the blow-up limit and the deflection angle $\theta_2-\theta_1$ of the free boundary near the origin.
	\end{rem}
	The present paper is built up as follows. In Section 2, based on Weiss' monotonicity formula, we investigate three different possible singular profiles close to the Type 1, Type 2 and Type 3 stagnation points respectively, that must be corner wave, cusp or flat profile. By analyzing the local cusp geometry at stagnation points, it is proven that cusp profile does not exist in Section 3, while flat profile is eliminated using the idea of perturbation in frequency formula in Section 4. Section 5 contains asymptotic  behavior of the free boundary approaching stagnation points.

	\section{Monotonicity formula and blow-up limits}
	In this section, we will investigate the possible blow-up limits and weighted densities near three types of stagnation points respectively.
	
	\subsection{Case 1. Type 1 stagnation point}	
	\
	\newline\indent
	For Subcase 1.1, the gradient function on the free boundary is 
	\begin{equation}
		\label{eq42}
		h_0(x,y)=(-x)^{\alpha}(-y)^{\beta},\,\,\alpha\geq0,\beta\geq1.
	\end{equation}
	
	We work out the explicit blow-up limits by means of the monotonicity formula. Monotonicity formula is important for the study of this paper, which is used to  determine that the blow-up limit is a homogeneous function. This concept has been confirmed in \cite{VARVAR2011}.
	The difference between the monotonicity formula for Case 1 and 2 and which in  \cite{VARVAR2011} is that there will be an extra perturbation term. The previous monotonicity formula is therefore not directly applicable. The derivative of the  Weiss adjusted boundary energy	we obtained  is  a nonnegative term plus a perturbation term.
	Let us first introduce $\delta:=\min\{X_0,\rm{dist}$$(X_0,\p\Omega)/2\}$ throughout the paper.
	\begin{lemma}
		(Monotonicity formula). Let $u$ be a weak solution of the free boundary problem \eqref{eq1} and for any $r\in (0,\delta)$, recalling \eqref{eq66}, 
		\begin{equation*}
			\begin{aligned}
				M_1(r)=r^{-2-\beta}\int_{B_{r}(X_0)}\left(\left |\nabla u\right |^{2}+(-x)^{\alpha}(-y)^{\beta}\chi_{\{u>0\}}\right)dX-\dfrac{\beta+2}{2}r^{-3-\beta}\int_{\p B_{r}(X_0)}u^{2}d\mathcal{H}^{1}.
			\end{aligned}
		\end{equation*}
		Then, the function $M_1(r)$ is differentiable almost everywhere on $(0,\delta)$ and for almost every $r\in (0,\delta)$, we have 
		\begin{equation}
			\label{eq2}
			\dfrac{dM_1(r)}{dr}=
			2r^{-2-\beta}\int_{\p B_{r}(X_0)}\left(\nabla u\cdot\nu-\frac{\beta+2}{2}\frac{u}{r}\right)^2d\mathcal{H}^{1}+h_1(r),
		\end{equation}
		where $\nu$ is the unit outer normal vector to the boundary and
		\begin{equation}
			\label{eq98}
			h_1(r)=r^{-3-\beta}\int_{B_{r}(X_0)}\alpha(-x)^{\alpha-1}(x_0-x)(-y)^{\beta}\chi_{\{u>0\}}dX.
		\end{equation}
	\end{lemma}
	\begin{proof}
		For convenience,  $M_1(r)$ can be written as $M_1(r)=I_1(r)-\frac{\beta+2}{2}J_1(r)$, where
		\begin{equation*}
			\begin{aligned}
				I_1(r)=&r^{-2-\beta}\int_{B_{r}(X_0)}\left(\left |\nabla u\right |^{2}+(-x)^{\alpha}(-y)^{\beta}\chi_{\{u>0\}}\right)dX,
			\end{aligned}
		\end{equation*}
		and
		\begin{equation*}
			\begin{aligned}
				J_1(r)=&r^{-3-\beta}\int_{\p B_{r}(X_0)}u^{2}d\mathcal{H}^{1}.
			\end{aligned}
		\end{equation*}
		Differentiating $J_1(r)$ with respect to $r$ gives
		\begin{equation*}
			\begin{aligned}
				\dfrac{dJ_1(r)}{dr}=(-3-\beta)r^{-4-\beta}\int_{\p B_{r}(X_0)}u^{2}d\mathcal{H}^{1}+r^{-3-\beta}\dfrac{d}{dr}\left(\int_{\p B_{r}(X_0)}u^{2}d\mathcal{H}^{1}\right),
			\end{aligned}
		\end{equation*}
		substituting the second term on the right side of the last equation by scaling properties, $ J'_1(r)$ can be derived as follows
		\begin{equation}
			\label{eq76}
			\begin{aligned}
				\frac{dJ_1(r)}{dr}=-(\beta+2)r^{-4-\beta}\int_{\p B_{r}(X_0)}u^{2}d\mathcal{H}^{1}+2r^{-3-\beta}\int_{\p B_{r}(X_0)}u\nabla u\cdot \nu d\mathcal{H}^{1}.	\end{aligned}
		\end{equation}
		A similar computation as Theorem 3.5 in \cite{VARVAR2011} gives rise to 
		\begin{equation}
			\label{eq45}
			\frac{dI_1(r)}{dr}
			=2r^{-2-\beta}\int_{\p B_{r}(X_0)}(\nabla u\cdot \nu)^2d\mathcal{H}^{1}-(2+\beta)r^{-3-\beta}\int_{\p B_{r}(X_0)}u\nabla u\cdot\nu d\mathcal{H}^1+h_1(r).
		\end{equation}
		Together with \eqref{eq76}, the proof is concluded.
	\end{proof}
	\begin{rem} It is important to recall that \eqref{eq28} implies that for any small $r<r_0$,
		\begin{equation}
			\label{eq69}
			u(X)\leq C_1 (-x)^{\frac{\alpha}{2}}(-y)^{\frac{\beta}{2}+1}
		\end{equation}
		holds	locally in $ B_r(X_0) $ for a positive constant $C_1$.
	\end{rem}
	In view of \eqref{eq69}, it is easy to check that $u_m\in W^{1,2}(B_r(X_0))$, hence there exists a weakly convergent subsequence of $u_m$ in $W^{1,2}(B_r(X_0))$.
	\begin{lemma} Let $u$ be a weak solution of the free boundary problem \eqref{eq1} satisfying \eqref{eq69} and                                                                               $r_m\rightarrow0+$ as $m\rightarrow\infty$,  then the limit
		$$\lim\limits_{m\rightarrow\infty}M_1(r_m)$$
		exists and is finite.
	\end{lemma}
	
	\begin{proof}
		Since $h_1(r)$ is integrable with respect to $r$ when $r\in(0,\delta)$, the function $\int_{0}^{r}h_1(s)ds$ is well-defined. By virtue of the monotonicity formula, we have
		$$\dfrac{d}{dr}\left( M_1(r)-\int_{0}^{r}h_1(s)ds\right) \geq0.$$
		As stated in  Theorem 1.1 of \cite{Spo},  $\lim\limits_{r\rightarrow0+}\left( M_1(r)-\int_{0}^{r}h_1(s)ds\right) $ exists. On account of \eqref{eq69} and $u(X_0)=0$, one has the existence of  $M_1(0+)$.
	\end{proof}
	
	With the help of the preceding two lemmas, we now use the Weiss' monotonicity formula to prove that the homogeneity of blow-up limit $u_0$, the higher compactness of blow-up sequence $u_m$ and the specific expression of $M_1(0+)$,  the results are given in Lemma 2.4.
	\begin{lemma}
		Let $u$ be a weak solution of the free boundary problem \eqref{eq1} with \eqref{eq69}, and $r_m $ be a vanishing sequence of positive real numbers such that the blow-up sequence $u_m(X)=r_m^{-\frac{\beta+2}{2}}u(X_0+r_mX)$ converges weakly in $W_{\text{loc}}^{1,2}(\mathbb{R}^2)$ to a blow-up limit $u_0$. Then 
		
		(i) $u_0(X)$ is a homogeneous function of $\frac{\beta+2}{2}$ degree;
		
		(ii) there is a subsequence $u_m$ which converges strongly to $u_0$ in $W_{\text{loc}}^{1,2}(\mathbb{R}^2)$ as $m\rightarrow\infty$;
		
		(iii) the weighted density is 
		\begin{equation}
			\label{eq43}
			M_1(0+)=(-x_0)^{\alpha}\lim\limits_{m     \rightarrow\infty}\int_{B_1}(y_-)^{\beta}\chi_{\{u_m>0\}}dX.
		\end{equation}
	\end{lemma}
	\begin{proof}
		(i) For each $0<\rho<\sigma<\delta$, integrating  \eqref{eq2} in the monotonicity formula with respect to $r$ on $(r_m\rho,r_m\sigma)$, we obtain
		\begin{equation*}
			\begin{aligned}
				&2\int_{B_{\sigma}\backslash B_{\rho}}|X|^{-\beta-4}\left(\nabla u_m(X)\cdot X-\dfrac{\beta+2}{2}u_m(X)\right)^2dX\\
				=&M_1(r_m\sigma)-M_1(r_m\rho)-\int_{r_m\rho}^{r_m\sigma}h_1(r)dr\rightarrow0,\,\,\,\text{as}\,\,m\rightarrow\infty,
			\end{aligned}     
		\end{equation*}
		the above equation is derived also due to the fact that $h_1(r)$ is integrable with respect to $r$ in $(0,\delta)$.
		
		It follows from the fact  that $u_m(X)\rightarrow u_0(X)$ weakly in $W^{1,2}_{\text{loc}}(\mathbb{R}^2)$ that
		\begin{equation}
			\label{eq22}
			\nabla u_0(X)\cdot X-\dfrac{\beta+2}{2}u_0(X)=0\,\,\,\text{a.e. in}\,\,\mathbb{R}^2,
		\end{equation}
		the homogeneity we wanted can be easily derived from \eqref{eq22}.
		
		(ii) 	We start by noticing that by the Bernstein estimate and the weak convergence of $u_m$,  $u_m\in W^{1,\infty}(B_r)$, for each $0<r<+\infty$. The strong convergence of the blow-up sequence in $L^2_{\text{loc}}(\mathbb{R}^2)$ can be derived by the compact embedding $W^{1,\infty}(U)\hookrightarrow L^{\infty}(U)\hookrightarrow L^2(U)$, where $U$ is a bounded domain of $ \mathbb{R}^2 $.
		
		Now, we claim that
		\begin{equation}
			\label{eq32}
			\limsup\limits_{m\rightarrow\infty}\left \|  \nabla u_m\right \|_{L^2(U)}\leq\left \| \nabla u_0\right \|_{L^2(U)}.
		\end{equation}

		In fact, for each $m$, the blow-up sequence $u_m$ satisfies 
		\begin{equation}
			\label{eq74}
			\Delta u_m=0\,\,\, \text{in }\, B_{\delta/r_m}\cap\{u_m>0\}.
		\end{equation} 
		Taking $m\rightarrow\infty$, 	 $u_m$ converges locally uniformly to $u_0$ in $\{u_m>0\}$ which gives that
		\begin{equation}
			\label{eq33}
			\Delta u_0=0\,\,\,\text{in }\,\{u_0>0\}.
		\end{equation}

		Besides, we give the following definition by the idea of mollification, for any $\varepsilon>0$,
		\begin{equation*}
			\begin{aligned}
				\operatorname{max}\left( u_m-\varepsilon,0\right) ^{1+\varepsilon}=\begin{cases}
					(u_m-\varepsilon)^{1+\varepsilon}\,\,&u_m>\varepsilon,\\
					0&0\leq u_m<\varepsilon.
				\end{cases}
			\end{aligned}
		\end{equation*}
		Noticing that $\operatorname{supp}\left( \operatorname{max}\left( u_m-\varepsilon,0\right) ^{1+\varepsilon}\right) \subset\{u_m>0\}$, we have 
		\begin{equation}
			\begin{aligned}	
				\label{eq21}
				\int_{\mathbb{R}^2}\nabla u_m\cdot\nabla\operatorname{max}\left( u_m-\varepsilon,0\right) ^{1+\varepsilon}\eta dX
				=-\int_{\mathbb{R}^2}u_m\nabla\operatorname{max}\left( u_m-\varepsilon,0\right) ^{1+\varepsilon}\cdot\nabla \eta d\mathcal{H}^{1},
			\end{aligned}
		\end{equation}
		for any $\eta\in C_0^1(\mathbb{R}^2)$.
		
		Let $\varepsilon\rightarrow0$, together with the convergence of $u_m$, we obtain
		\begin{equation}
			\label{eq75}
			\lim\limits_{m\rightarrow\infty}\int_{\mathbb{R}^2}|\nabla u_m|^2\eta dX=\int_{\mathbb{R}^2}|\nabla u_0|^2\eta dX,
		\end{equation}
		the claim \eqref{eq32} immediately follows from the equation \eqref{eq75}. Eventually, by means of Proposition 3.32 in \cite{Brezis},
		we  can  show that there exists a subsequence still relabeled by $u_m$ such that $\nabla u_m\rightarrow\nabla u_0$ strongly in $L^2_{\text{loc}}(\mathbb{R}^2).$
		
		(iii) 	For any $m\rightarrow\infty$, $r_m $ is an infinitesimal sequence of positive real numbers, we calculate the limit of $M_1(r_m)$ directly, 
		\begin{equation}
			\label{eq78}
			\begin{aligned}
				M_1(0+)=\lim\limits_{m\rightarrow\infty}&M_1(r_m)\\
				=\lim\limits_{m\rightarrow\infty}&\bigg[ \int_{B_1}\left( |\nabla u_m|^2+(-x_0)^{\alpha}(y_-)^{\beta}\chi_{\{u_m>0\}}\right) dX-\dfrac{\beta+2}{2}\int_{\p B_1}u_m^2d\mathcal{H}^1\bigg]\\
				=(-x_0&)^{\alpha}\lim\limits_{m\rightarrow\infty}\int_{B_1}(y_-)^{\beta}\chi_{\{u_m>0\}}dX+ \int_{B_1}|\nabla u_0|^2dX-\dfrac{\beta+2}{2}\int_{\p B_1}u_0^2d\mathcal{H}^1,
			\end{aligned}
		\end{equation}
		the last equality being a consequence of the strong convergence and trace theorem.
		
		On the other hand, the homogeneity of  $u_0$ suggests that
		\begin{equation}
			\label{eq77}
			\int_{B_1}|\nabla u_0|^2dX=\dfrac{\beta+2}{2}\int_{\p B_1}u_0^2d\mathcal{H}^1.
		\end{equation}
		Accordingly, 	substituting \eqref{eq77} into \eqref{eq78} yields the conclusion \eqref{eq43}.
	\end{proof}
	
	It is apparent from	\eqref{eq43}  that the result of $M_1(0+)$ is decided by the limit of characteristic function $\lim\limits_{m\rightarrow\infty}\chi_{\{u_m>0\}}$. In the next proposition, we deduce that $\lim\limits_{m\rightarrow\infty}\chi_{\{u_m>0\}}$ must be $$\chi_{\left\lbrace-\frac{\beta+4}{2(\beta+2)}\pi <\theta<-\frac{\beta}{2(\beta+2)}\pi\right\rbrace }\,\,\text{or}\,\, 0\,\,\text{or}\,\,1,$$ thereby giving all possible exact forms of the blow-up limits and the corresponding densities for Subcase 1.1.
	\begin{prop}
		Let $u$ be a weak solution of the free boundary problem \eqref{eq1} satisfying \eqref{eq69}. Then  we may decompose $S^u$ into the disjoint union $S^u=S^u_{1,\text{corner}}\cup S^u_{1,\text{cusp}}\cup S^u_{1,\text{flat}}$, we define 
		
		(i) $S^u_{1,\text{corner}}=\left\lbrace X=(x_0,0)\in S^u\Big| M_1(0+)=(-x_0)^{\alpha}\int_{B_1\bigcap\left\lbrace -\frac{\beta+4}{2(\beta+2)}\pi <\theta<-\frac{\beta}{2(\beta+2)}\pi\right\rbrace } (y_-)^{\beta}dX\right\rbrace $, and the corresponding  blow-up limit is 
		\begin{equation*}
			\begin{aligned}
				u_0(r\cos\theta,r\sin\theta)
				=\dfrac{2(-x_0)^{\frac{\alpha}{2}}}{\beta+2}r^\frac{\beta+2}{2}\cos^{\frac{\beta}{2}}\left( \dfrac{\pi}{\beta+2}\right)  \cos\left( \dfrac{\beta+2}{2}\theta+\dfrac{\beta+2}{4}\pi\right) \chi_{\left\lbrace-\frac{\beta+4}{2(\beta+2)}\pi <\theta<-\frac{\beta}{2(\beta+2)}\pi\right\rbrace },
			\end{aligned}
		\end{equation*}
		the corner wave profile see Fig. 6(a);
		
		(ii)	$S^u_{1,\text{cusp}}=\left\lbrace X=(x_0,0)\in S^u|M_1(0+)=0\right\rbrace $, and the corresponding  blow-up limit is 
		$$u_0\equiv0,$$
		the cusp wave profile see Fig. 11(a) and (b);
		
		(iii)	$S^u_{1,\text{flat}}=\left\lbrace X=(x_0,0)\in S^u\big|M_1(0+)=(-x_0)^{\alpha}\int_{B_1} (y_-)^{\beta}dX\right\rbrace $, and the corresponding  blow-up limit is 
		$$u_0\equiv0,$$
		the flat wave profile see Fig. 11(c).
		\begin{figure}[!h]
			\includegraphics[width=150mm]{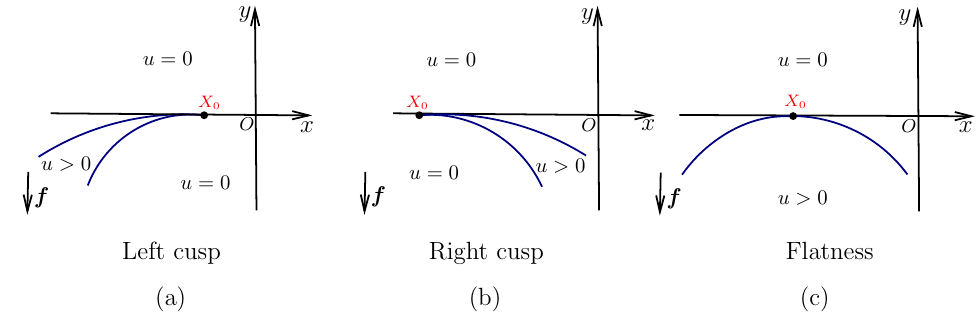}	
			\caption{Cusp point and flat  point }
		\end{figure}
	\end{prop}
	\begin{proof}
		Setting $\phi=(\phi_1,\phi_2)\in C^1_0(\mathbb{R}^2,\mathbb{R}^2)$ as the test function in \eqref{eq23}, then we obtain  
		\begin{equation*}
			\begin{aligned}
				0=&\int_{\Omega_m}|\nabla u_m|^2\operatorname{div}\phi dZ-2\int_{\Omega_m}\nabla u_m D\phi\nabla u_mdZ
				+\int_{\Omega_m}(-x_0-r_mz_1)^{\alpha}(z_{2-})^{\beta}\operatorname{div}\phi\chi_{\{u_m>0\}}dZ\\
				&+\int_{\Omega_m}\big(-\alpha(-x_0-r_mz_1)^{\alpha-1}r_m(z_{2-})^{\beta}\phi_1-\beta(-x_0-r_mz_1)^{\alpha}(z_{2-})^{\beta-1}\phi_2\big)\chi_{\{u_m>0\}}dZ,
			\end{aligned}
		\end{equation*}
		where  $\Omega_m=\{Z=(z_1,z_2)|X_0+r_mZ\in\Omega\}$. Passing to the limit as  $ m\rightarrow \infty$, we get that
		\begin{equation}
			\label{eq7}
			\begin{aligned}
				0=&\int_{\mathbb{R}^2}\left(|\nabla u_0|^2\operatorname{div}\phi-2\nabla u_0D\phi\nabla u_0\right)dX	+\int_{\mathbb{R}^2}(-x_0)^{\alpha}(y_-)^{\beta}\operatorname{div}\phi\chi_0dX\\
				&+\int_{\mathbb{R}^2}\big(-\beta(-x_0)^{\alpha}(y_-)^{\beta-1}\phi_2\big)\chi_0dX,
			\end{aligned}
		\end{equation}
		where $\chi_0$ is the strong $L^1_{\text{loc}}$-limit of $\chi_{\{u_m>0\}}$ along a subsequence. The values of the function $\chi_0$ are almost everywhere in $\{0,1\}$, and the locally uniform convergence of $u_m$ to $u_0$ implies that 
		\begin{equation}
			\label{eq24}
			\chi_{0}=1\text{ in }\{u_0>0\}.
		\end{equation}
		
		Suppose first that the set $\{u_0>0\}$ is   empty, then \eqref{eq7} tells us that in the interior of $\{u_0=0\}$, 
		\begin{equation}
			\label{eq79}
			(-x_0)^{\alpha}\int_{\mathbb{R}^2}\operatorname{div}\big((y_-)^{\beta}\phi\big)\chi_0dX=0.
		\end{equation}
		It follows from \eqref{eq79} that   $\chi_0$ is a constant in each connected component of the interior of $\{u_0=0\}\cap\{X=(x,y)|y<0\}$, and we  denote it  as $\overline{\chi}_0$. Thus it is sufficient to show that when the blow-up limit $u_0\equiv0$, the corresponding density is
		$$M_1(0+)= 0\,\,\text{or}\,\,M_1(0+)=(-x_0)^{\alpha}\int_{B_1} (y_-)^{\beta}dX.$$
		
		Consider now the second case when  $\{u_0>0\}$ is not an empty set.
		
		Together with the homogeneity of blow-up limit	$u_0$, we can write the harmonic function $u_0$ in polar coordinates, we have that 
		\begin{equation}
			\label{eq80}
			u_0(r,\theta)=C_0r^{\frac{\beta+2}{2}}\cos\left( \dfrac{\beta+2}{2}\theta+\varphi_0\right),
		\end{equation}
		where $C_0>0$ and $\varphi_0$ are constants to be determined. Each connected component of $\{u_0>0\}$ is thereby a cone with vertex at the origin and an  angle of $\frac{2\pi}{\beta+2}$. 
		
		If we 	choose any point  $Z\in\p\{u_0>0\}\backslash\{(0,0)\}$, there is a sufficiently small constant $\omega$ such that the outer normal to $\p\{u_0>0\}$ is a constant vector $\nu(Z)$ in $B_{\omega}(Z)\cap\p\{u_0>0\}$. Setting $\phi(X):=\eta(X)\nu(Z)$ for any $\eta\in C_0^1(B_\omega(Z))$ as the test function in \eqref{eq7}, combining with \eqref{eq24}, we get that
		\begin{equation}
			\label{eq25}
			\begin{aligned}
				\int_{B_\omega(Z)\cap\p\{u_0>0\}}|\nabla u_0|^2\eta d\mathcal{H}^1=(-x_0)^{\alpha}\int_{B_\omega(Z)\cap\p\{u_0>0\}}(y_-)^{\beta}(1-\overline{\chi_0})\eta d\mathcal{H}^1.
			\end{aligned}
		\end{equation}
		Indeed, it follows from Hopf's lemma that $|\nabla u_0|\neq0$ on $B_\omega(Z)\cap\p\{u_0>0\}$, from which and \eqref{eq25}, we deduce 
		\begin{equation}
			\label{eq26}
			\begin{aligned}
				|\nabla u_0|^2=(-x_0)^{\alpha}(y_-)^{\beta}\,\,\text{on}\,\,\p\{u_0>0\}.
			\end{aligned}
		\end{equation}
		
		Now, we will establish the possible explicit form of $u_0$ with the help of \eqref{eq26}. Applying \eqref{eq80} in \eqref{eq26}, we have that
		\begin{equation}
			\label{eq27}
			\dfrac{(\beta+2)^2}{4}C_0^2=(-x_0)^{\alpha}\left( \left( \sin\theta\right)_-\right)  ^{\beta},
		\end{equation}
		which holds on both boundaries of the cone.
		
		We denote	the boundaries of the cone intersect  $ x $-axis at  angles $\theta_1$ and $\theta_2$ respectively, and $\theta_1<\theta_2$. Besides, $x_0<0$ is mentioned previously, then the following relationships hold,
		\begin{center}
			$	\sin\theta_1=\sin\theta_2<0$ and $\theta_2-\theta_1=\dfrac{2\pi}{\beta+2}$.
		\end{center}
		Thus one has that $ \theta_1=-\frac{\pi}{\beta+2}-\frac{\pi}{2}=-\frac{\beta+4}{2(\beta+2)}\pi$ and $ \theta_2=\frac{\pi}{\beta+2}-\frac{\pi}{2}=-\frac{\beta}{2(\beta+2)}\pi $, which yields that
		$$C_0=\dfrac{2(-x_0)^{\frac{\alpha}{2}}}{\beta+2}\cos^{\frac{\beta}{2}}\left( \dfrac{\pi}{\beta+2}\right) .$$
		
		Moreover, it follows from the boundary condition that $\varphi_0=\frac{\beta+2}{4}\pi+k\pi$, for some integer $k$. Noticing that $\varphi_0=\frac{\beta+2}{4}\pi$ and $u_0(r,\theta)\geq0$ for $k=0$. Finally, the expression of $u_0$ is
		\begin{equation*}
			\begin{aligned}
				u_0(r,\theta)
				=\dfrac{2(-x_0)^{\frac{\alpha}{2}}}{\beta+2}r^\frac{\beta+2}{2}\cos^{\frac{\beta}{2}}\left( \dfrac{\pi}{\beta+2}\right)  \cos\left( \dfrac{\beta+2}{2}\theta+\dfrac{\beta+2}{4}\pi\right) \chi_{\left\lbrace-\frac{\beta+4}{2(\beta+2)}\pi <\theta<-\frac{\beta}{2(\beta+2)}\pi\right\rbrace },
			\end{aligned}
		\end{equation*}
		and	the corresponding weighted density is 
		$$M_1(0+)=(-x_0)^{\alpha}\int_{B_1\bigcap\left\lbrace -\frac{\beta+4}{2(\beta+2)}\pi <\theta<-\frac{\beta}{2(\beta+2)}\pi\right\rbrace } (y_-)^{\beta}dX,$$
		which concludes the proof of	Proposition 2.5.
	\end{proof}

	Reasoning as above, we just need to substitute the equation \eqref{eq42} as follows respectively,
	\begin{center}
		$h_0(x,y)=x^{\alpha}y^{\beta}\,\,$ for Subcase 1.2;
	\end{center}
	\begin{center}
		$h_0(x,y)=(-x)^{\alpha}y^{\beta}\,\,$ for Subcase 1.3;
	\end{center}
	\begin{center}
		$h_0(x,y)=x^{\alpha}(-y)^{\beta}\,\,$ for Subcase 1.4.
	\end{center}
	The pertinent results have been stated in Theorem 1.1 without proof.
	\subsection{Case 2. Type 2 stagnation point} 
	\
	\newline \indent 
	We consider the gradient function in Subcase 2.1 as 
	\begin{equation}
		\label{eq44}
		h_0(x,y)=(-x)^{\alpha}(-y)^{\beta},\,\,\alpha\geq1,\beta\geq0.
	\end{equation}
	Similar to what we did in the proof of Subcase 1.1, we can obtain the following conclusions of Subcase 2.1.
	\begin{rem}
		Note that \eqref{eq28} implies that for any small $r<r_0$, 
		\begin{equation}
			\label{eq94}
			u(X)\leq C_2 (-x)^{\frac{\alpha}{2}+1}(-y)^{\frac{\beta}{2}}
		\end{equation}
		locally in $B_r(X_0)$ for a positive constant $C_2$.
	\end{rem}
	\begin{prop}
		Suppose that \eqref{eq94} holds, and let $u$ be a weak solution of the free boundary problem \eqref{eq1}, then the following conclusions hold.
		
		(i)	(Monotonicity formula) For any $r\in (0,\delta)$, recalling \eqref{eq66}, the Weiss adjusted boundary energy is 
		\begin{equation}
			\label{eq81}
			\begin{aligned}
				M_2(r)=r^{-2-\alpha}\int_{B_{r}(X_0)}\left(\left |\nabla u\right |^{2}+(-x)^{\alpha}(-y)^{\beta}\chi_{\{u>0\}}\right)dX-\dfrac{\alpha+2}{2}r^{-3-\alpha}\int_{\p B_{r}(X_0)}u^{2}d\mathcal{H}^{1}.
			\end{aligned}
		\end{equation}
		Then 	the function $M_2(r)$ is differentiable almost everywhere on $(0,\delta)$ and for a.e. $r\in (0,\delta)$, we have 
		\begin{equation*}
			\dfrac{	dM_2(r)}{dr}=
			2r^{-2-\alpha}\int_{\p B_{r}(X_0)}\left(\nabla u\cdot\nu-\frac{\alpha+2}{2}\frac{u}{r}\right)^2d\mathcal{H}^{1}+h_2(r),
		\end{equation*}
		where 
		$$h_2(r)=r^{-3-\alpha}\int_{B_{r}(X_0)}\beta(-x)^{\alpha}(y_0-y)(-y)^{\beta-1}\chi_{\{u>0\}}dX.$$
		
		(ii) (Blow-up sequence and blow-up limit) For  $m\rightarrow\infty$,	let $r_m $ be a vanishing sequence of positive real numbers such that the blow-up sequence $u_m(X)=r_m^{-\frac{\alpha+2}{2}}u(X_0+r_mX)$ converges weakly in $W_{\text{loc}}^{1,2}(\mathbb{R}^2)$ to a blow-up limit $u_0$. Then $u_0$ is a homogeneous function of $\frac{\alpha+2}{2}$ degree. Moreover, $u_m$ converges strongly to $u_0$ in $W_{\text{loc}}^{1,2}(\mathbb{R}^2)$.

		(iii) (Weighted  density) The limit $M_2(0+)$ exists and the weighted density is  $$M_2(0+)=(-y_0)^{\beta}\lim\limits_{m      \rightarrow\infty}\int_{B_1}(x_-)^{\alpha}\chi_{\{u_m>0\}}dX.$$
		
		(iv) (Possible singular profiles) The set of corner wave stagnation point is  $$S^u_{2,\text{corner}}=\left\lbrace X=(0,y_0)\in S^u\bigg|M_2(0+)=(-y_0)^{\beta}\int_{B_1\bigcap\left\lbrace\frac{\alpha+1}{\alpha+2}\pi<\theta<\frac{\alpha+3}{\alpha+2}\pi\right\rbrace } (x_-)^{\alpha}dX\right\rbrace ,$$ and the corresponding  blow-up limit is 
		\begin{equation*}
			\begin{aligned}
				u_0(r,\theta)
				=\dfrac{2(-y_0)^{\frac{\beta}{2}}}{\alpha+2}r^\frac{\alpha+2}{2}\cos^{\frac{\alpha}{2}}\left( \frac{\alpha+1}{\alpha+2}\pi\right) \cos\left( \dfrac{\alpha+2}{2}\theta+\dfrac{2-\alpha}{2}\pi\right) \chi_{\left\lbrace \frac{\alpha+1}{\alpha+2}\pi<\theta<\frac{\alpha+3}{\alpha+2}\pi\right\rbrace }.
			\end{aligned}
		\end{equation*}
		
		The set of cusp stagnation point and flat stagnation point are   $$S^u_{2,\text{cusp}}=\left\lbrace X=(0,y_0)\in S^u\big|M_2(0+)=0\right\rbrace$$ and $$S^u_{2,\text{flat}}=\left\lbrace X=(0,y_0)\in S^u\Big|M_2(0+)=(-y_0)^{\beta}\int_{B_1} (x_-)^{\alpha}dX\right\rbrace $$ respectively, and the corresponding  blow-up limit of the both is 
		$$u_0\equiv0.$$
		The corner wave profile see Fig. 8(a), the cusp wave profile see Fig. 12(a) and (b), the flat wave profile see Fig. 12(c).
		\begin{figure}[h]
			\includegraphics[width=130mm]{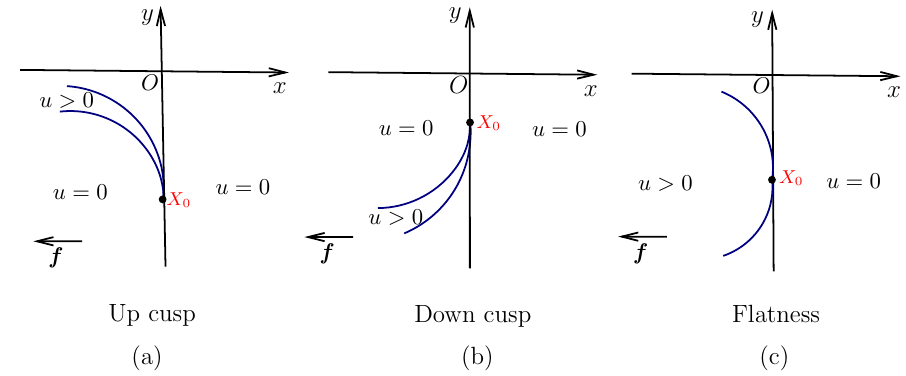}	
			\caption{Cusp point  and flat point}
		\end{figure}
	\end{prop}
	
	\subsection{Case 3. Type 3 stagnation point}
	\
	\newline\indent
	In this subsection, we would like to study the singular profile near Type 3 stagnation point $X_0=(0,0)$. 
	Similar to what we did in Lemma 2.1, the derivative of the Weiss functional for the Type 3 stagnation point will be concluded firstly.
	\begin{lemma}
		(Monotonicity formula). Let $u$ be a weak solution of \eqref{eq1} and  for any $r\in (0,\delta)$, recalling \eqref{eq66}  we have that
		\begin{equation*}
			\begin{aligned}
				M_3(r)=&r^{-2-\alpha-\beta}\int_{B_{r}(X_0)}\left(\left |\nabla u\right |^{2}+|x|^{\alpha}|y|^{\beta}\chi_{\{u>0\}}\right)dX-\dfrac{\alpha+\beta+2}{2}r^{-3-\alpha-\beta}\int_{\p B_{r}(X_0)}u^{2}d\mathcal{H}^{1}.
			\end{aligned}
		\end{equation*}
		Then,	the function $M_3(r)$ is differentiable almost everywhere on $(0,\delta)$ and for almost everywhere $r\in (0,\delta)$, we have 
		\begin{equation*}
			\dfrac{dM_3(r)}{dr}=
			2r^{-2-\alpha-\beta}\int_{\p B_{r}(X_0)}\left(\nabla u\cdot\nu-\frac{\alpha+\beta+2}{2}\frac{u}{r}\right)^2d\mathcal{H}^{1}.
		\end{equation*}
	\end{lemma}
	\begin{rem}
		It is important to recall that the decay rate \eqref{eq28} implies that for any small $r<r_0$,
		\begin{equation}
			\label{eq82}
			u(X)\leq C_3\left( |x|^{\frac{\alpha}{2}+1}|y|^{\frac{\beta}{2}}+|x|^{\frac{\alpha}{2}}|y|^{\frac{\beta}{2}+1}\right)
		\end{equation}
		holds	locally in $B_r(X_0)$ for a positive constant $C_3$.
	\end{rem}
	
	\begin{lemma}
		Suppose that  \eqref{eq82} holds, and let $u$ be a weak  solution of the free boundary problem \eqref{eq1}. For  $m\rightarrow\infty$,	let $r_m $ be a vanishing sequence of positive real numbers. 
		
		(i) (Blow-up sequence and blow-up limit) Suppose that the blow-up sequence $u_m(X)=r_m^{-\frac{\alpha+\beta+2}{2}}u(X_0+r_mX)$ converges weakly in $W_{\text{loc}}^{1,2}(\mathbb{R}^2)$ to a blow-up limit $u_0$. Then $u_0$ is a homogeneous function of $\frac{\alpha+\beta+2}{2}$ degree and $u_m$ converges strongly to $u_0$ in $W_{\text{loc}}^{1,2}(\mathbb{R}^2)$.

		(ii) (Weighted  density) The limit $M_3(0+)$ exists and the weighted density is  $$M_3(0+)=\lim\limits_{m      \rightarrow\infty}\int_{B_1} |x|^{\alpha}|y|^{\beta}\chi_{\{u_m>0\}}dX.$$
	\end{lemma}
	\begin{proof}
		The existence of $M_3(0+)$ follows directly from that $M'_3(r)\geq0$ and $|M_3(r)|$ is bounded. The same method of Lemma 2.4 can be used to verify the  rest conclusion of this lemma, and we omit the proof here.
	\end{proof}
	
	\begin{prop}
		Let $u$ be a weak solution of the free boundary problem \eqref{eq1} satisfying \eqref{eq82}, then we have the following conclusion:
		
		The set of corner wave stagnation point is  $$S^u_{3,\text{corner}}=\left\lbrace X\in S^u\bigg|M_3(0+)=\int_{B_1\cap\left\lbrace \theta_1 <\theta<\theta_2\right\rbrace } |x|^{\alpha}|y|^{\beta}dX\right\rbrace ,$$ 
		where $\theta_1$ and $\theta_2$ are the  angles of the asymptotic direction of free boundaries intersect $x$-axis near the origin with $\theta_1<\theta_2$,  $\theta_2-\theta_1=\frac{2\pi}{\alpha+\beta+2}$ and satisfying the boundary condition of blow-up limit. For each pair of $\theta_1$ and $\theta_2$, the corresponding  blow-up limit is 
		\begin{equation*}
			\begin{aligned}
				u_0(r,\theta)
				=\dfrac{2|\cos\theta_1|^{\frac{\alpha}{2}}|\sin\theta_1|^{\frac{\beta}{2}}}{\alpha+\beta+2}r^\frac{\alpha+\beta+2}{2}  \cos\left( \frac{\alpha+\beta+2}{2}\theta+\arcsin\left|\cos\left( \frac{\alpha+\beta+2}{2}\theta_1\right)  \right| \right) \chi_{\left\lbrace\theta_1 <\theta<\theta_2\right\rbrace }.
			\end{aligned}
		\end{equation*}
		
		The set of cusp stagnation point and flat stagnation point are  $$S^u_{3,\text{cusp}}=\left\lbrace X\in S^u\big|M_3(0+)=0\right\rbrace$$ and  $$S^u_{3,\text{flat}}=\left\lbrace X\in S^u\Big|M_3(0+)=\int_{B_1} |x|^{\alpha}|y|^{\beta}dX\right\rbrace$$ respectively, and the corresponding  blow-up limit of the both is 
		$$u_0\equiv0.$$
		The corner wave profile see Fig. 10, the cusp profile see Fig. 13(a) and (b), and flat profile see Fig. 13(c).
		\begin{figure}[h]
			\includegraphics[width=140mm]{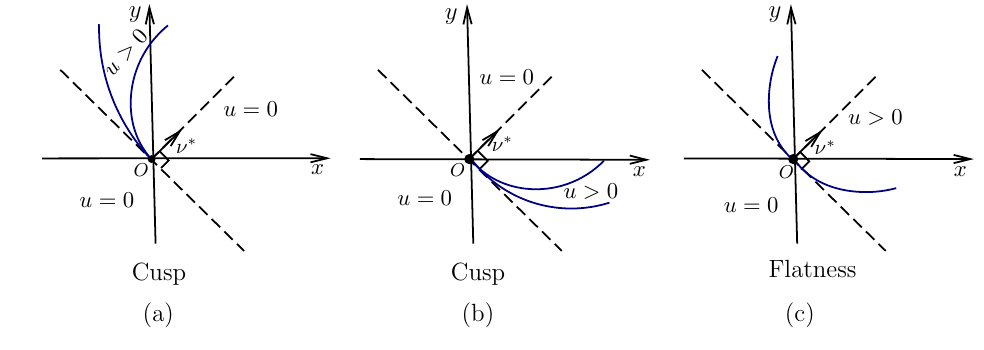}	
			\caption{Cusp point and flat point}
		\end{figure}
	\end{prop}
	\begin{proof}
		By suitable modification to the proof of Proposition 2.5, we can confirm the cases of cusp and flatness. Furthermore, under the assumption that $\{u_0>0\}$ is non-empty,
		\begin{equation}
			\label{eq61}
			u_0(r,\theta)=r^{\frac{\alpha+\beta+2}{2}}C_0\cos\left( \dfrac{\alpha+\beta+2}{2}\theta+\varphi_0\right) 
		\end{equation}
		holds,	and each connected component of $\{u_0>0\}$ is a cone with vertex at the origin and an  angle of $\frac{2\pi}{\alpha+\beta+2}$. 
		It follows from \eqref{eq61} and the boundary condition that
		\begin{equation}
			\label{eq62}
			\dfrac{(\alpha+\beta+2)^2}{4}C_0^2=|\cos\theta|^{\alpha}|\sin\theta|^{\beta},
		\end{equation}
		which holds on both boundaries of the cone.
		
		In each connected component of $\{u_0>0\}$, we introduce a velocity potential $\phi_0$ defined by 
		\begin{center}
			$\dfrac{\p\phi_0}{\p x}=-\dfrac{\p u_0}{\p y}$ and $\dfrac{\p\phi_0}{\p y}=\dfrac{\p u_0}{\p x}.$
		\end{center}
		Let $\psi_0(x_1,x_2,x_3)=\phi_0\left( \sqrt{x_1^2+x_2^2},x_3\right) $, through a direct calculation,
		we have in $\{u_0>0\}$,
		\begin{equation}
			\begin{aligned}
				\Delta\psi_0=\dfrac{1}{x}\dfrac{\p}{\p x}\left( x\dfrac{\p\phi_0}{\p x}\right) +\dfrac{1}{x^2}\dfrac{\p^2\phi_0}{\p\gamma^2}+\dfrac{\p^2\phi_0}{\p y^2}=\dfrac{1}{x}\dfrac{\p\phi_0}{\p x},
			\end{aligned}
		\end{equation}	
		where $x=\sqrt{x_1^2+x_2^2},y=x_3,\gamma=\arctan\frac{x_2}{x_1}$. This implies that $\psi_0(x_1,x_2,x_3)$ in $\{u_0>0\}$ satisfies 
		\begin{equation}
			\label{eq92}
			\dfrac{1}{r^2}\dfrac{\p}{\p r}\left( r^2\dfrac{\p\psi_0}{\p r}\right) +\dfrac{1}{r^2\cos\theta}\dfrac{\p}{\p \theta}\left( \cos\theta\dfrac{\p\psi_0}{\p \theta}\right) -\dfrac{1}{r}\dfrac{\p\phi_0}{\p r}+\dfrac{\sin\theta}{r^2\cos\theta}\dfrac{\p\phi_0}{\p \theta}=0,
		\end{equation}
		where $r=\sqrt{x_1^2+x_2^2+x_3^2}$ and $\theta=\arctan\frac{x_3}{\sqrt{x_1^2+x_2^2}}$.
		
		Since $\phi_0(r\cos\theta,r\sin\theta)$ is a homogeneous function of $\frac{\alpha+\beta+2}{2}$  degree,  we can rewrite  that 
		\begin{equation}
			\label{eq91}
			\phi_0(r\cos\theta,r\sin\theta)=r^{\frac{\alpha+\beta+2}{2}}g(\cos\theta),
		\end{equation}
		$g(\cos\theta)$ is a smooth function of $\cos\theta$, to be determined later. For ease of notations, setting $z=\cos\theta$.
		Inserting \eqref{eq91} into \eqref{eq92}, we compute that $g(z)$ satisfies the following Tschebysheff equation in $\{u_0>0\}$,
		$$(1-z^2)g''(z)-zg'(z)+\frac{(\alpha+\beta+2)^2}{4}g(z)=0.$$
		The solution of the Tschebysheff equation (see in \cite{ZJLS}) is
		$$g(\cos\theta)=a\cos\left( \frac{\alpha+\beta+2}{2}\theta\right) +b\sin\left( \frac{\alpha+\beta+2}{2}\theta\right) ,$$
		where $a$ and $b$ are the coefficients to be determined.	  
		
		Noticing that $|\nabla u_0|^2=|\nabla\phi_0|^2$, we have 	
		\begin{equation}
			\label{eq63}
			\begin{cases} 
				\left( \frac{\alpha+\beta+2}{2}\right) ^2(a^2+b^2)=|\cos\theta_1|^{\alpha}|\sin\theta_1|^{\beta},\\
				\left( \frac{\alpha+\beta+2}{2}\right) ^2(a^2+b^2)=|\cos\theta_2|^{\alpha}|\sin\theta_2|^{\beta}.
			\end{cases}
		\end{equation}
		It follows from  the homogeneity of $u_0$ that 
		\begin{equation}
			u_0(r\cos\theta,r\sin\theta)=r^{\frac{\alpha+\beta+2}{2}}f(\cos\theta),
		\end{equation}
		where $f(\cos\theta)$ is a smooth function of $\cos\theta$, to be determined later, which	together with \eqref{eq91} gives that
		$$\dfrac{\p\phi_0}{\p\theta}=r\dfrac{\p u_0}{\p r},$$
		the relationship of $f(\cos\theta)$ and $g(\cos\theta)$ can be easily derived  as follows
		$$ \dfrac{dg(\cos\theta)}{d\theta} =\frac{\alpha+\beta+2}{2}f(\cos\theta).$$
		
		Moreover, recalling the free boundary condition of $u_0$,
		$$u_0(r\cos\theta_i,r\sin\theta_i)=r^{\frac{\alpha+\beta+2}{2}}f(\cos\theta_i)=0,\,i=1,2,$$
		thus we deduce that
		$$\dfrac{dg(\cos\theta)}{d\theta}\bigg| _{\theta=\theta_i}=0,$$
		a direct computation yields that
		\begin{equation}
			\label{eq93}
			a\sin\left( \frac{\alpha+\beta+2}{2}\theta_i\right) -b\cos\left( \frac{\alpha+\beta+2}{2}\theta_i\right)=0,\,\text{for}\,i=1,2.
		\end{equation}	
		
		All in all,  $\theta_i$	solves the equations \eqref{eq63} and \eqref{eq93} for $ i=1,2$. For simplicity of presentation, we define $A=\frac{2\pi}{\alpha+\beta+2}$ and the function $	H(\tan\theta)$ as 	\begin{equation}
			\label{eq90}
			H(\tan\theta)=\left| \cos A-\tan\theta\sin A\right| ^{\alpha}\left| \cos A+\dfrac{1}{\tan\theta}\sin A\right| ^{\beta}-1,
		\end{equation}
		solving for $\theta_1$ and $\theta_2$ in \eqref{eq63} and \eqref{eq93} is equivalent to solving for the zeros of the function $H$. By conveniently denoting $s(\theta)=\tan\theta$, we now write $H(\tan\theta)$ as 
		\begin{equation*}
			H(s)=\left| \cos A-s\sin A\right| ^{\alpha}\left| \cos A+\dfrac{1}{s}\sin A\right| ^{\beta}-1.
		\end{equation*}	
		Obviously, $H(s)\rightarrow+\infty$ if $s\rightarrow0$ or $\infty$. Besides $H(\frac{1}{\tan A})=-1$ and $H(-\tan A)=-1$. 
		The zeros of function \eqref{eq90} are proved via the monotonic properties of $H(s)$. The situation of the zeros is related to the size of $\alpha$ and $\beta$. Here we illustrate the case of $\alpha\geq\beta$ as an example, and the result of $\alpha<\beta$ is analogous. 
		We can summarize the results as follows. 
		
		If $0<\tan A\leq\frac{2\sqrt{\alpha\beta}}{\alpha-\beta}$, the graph of $H(s)$ is illustrated in  Fig. 14(a), 
		if $\tan A>\frac{2\sqrt{\alpha\beta}}{\alpha-\beta}$, the graph of $H(s)$ is illustrated in  Fig. 14(b).
		\begin{figure}[h]
			\includegraphics[width=130mm]{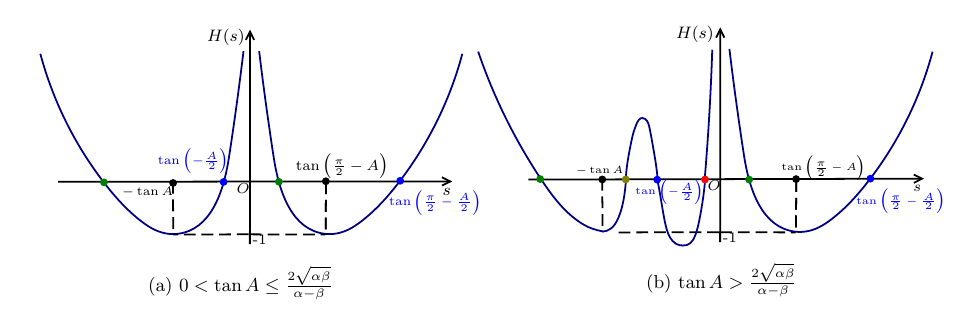}	
			\caption{The graph of $H(s)$ for $\alpha\geq\beta$}
		\end{figure}
		
		The zeros of $H(s)$ can imply the solutions of $\theta_1$ in \eqref{eq63}. 
		If $0<\tan A\leq\frac{2\sqrt{\alpha\beta}}{\alpha-\beta}$, there are eight pairs of solutions $(\theta_1,\theta_2)$. Furthermore, four pairs of solutions are symmetric about the coordinate axes, see Fig. 15, the other four pairs of them are non-symmetric, see Fig. 16.
		\begin{figure}[h]
			\includegraphics[width=137mm]{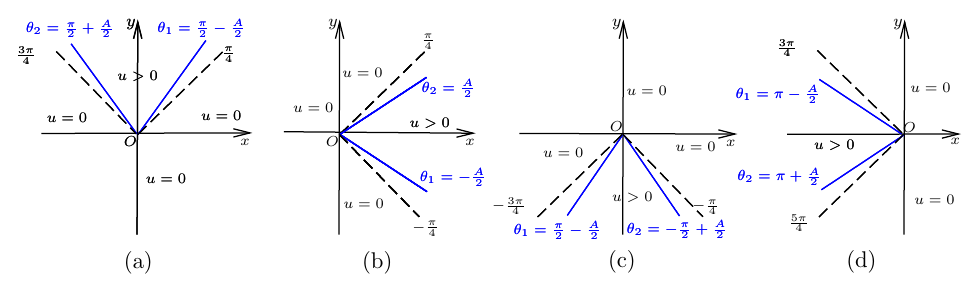}	
			\caption{The symmetric solutions  for $0<\tan A\leq\frac{2\sqrt{\alpha\beta}}{\alpha-\beta}$ ($\alpha\geq\beta$)}
		\end{figure}
		
		\begin{figure}[h]
			\includegraphics[width=137mm]{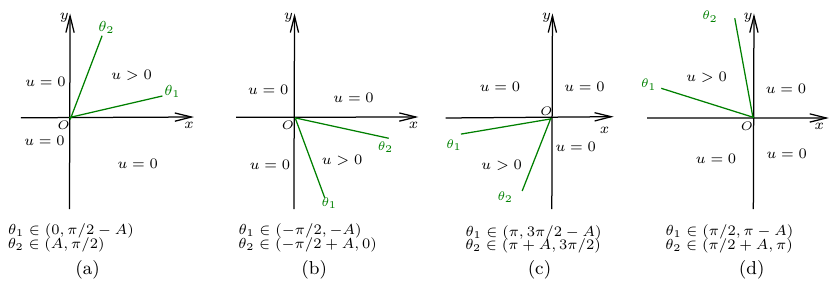}	
			\caption{The non-symmetric solutions for $0<\tan A\leq\frac{2\sqrt{\alpha\beta}}{\alpha-\beta}$ ($\alpha\geq\beta$)}
		\end{figure}
		
		If $\tan A>\frac{2\sqrt{\alpha\beta}}{\alpha-\beta}$, except for the solutions mentioned in Fig. 15 and Fig. 16, there are other four pairs of non-symmetric solutions, see Fig. 17.
		\begin{figure}[h]
		\includegraphics[width=150mm]{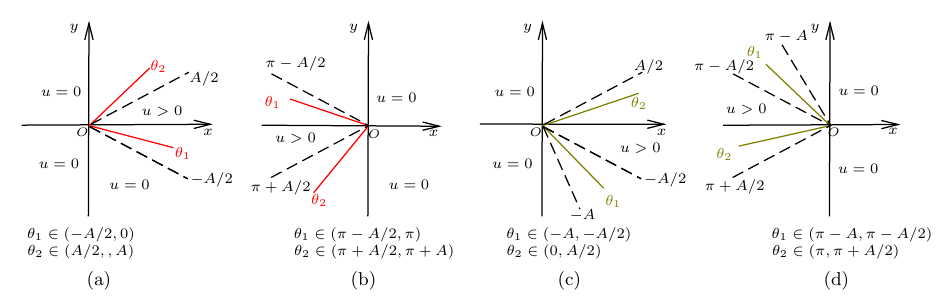}	
			\caption{ Additional non-symmetric solutions for $\tan A>\frac{2\sqrt{\alpha\beta}}{\alpha-\beta}$ ($\alpha\geq\beta$)}
		\end{figure}
		
		Naturally, we can decide the coefficients $C_0$ and $\varphi_0$ by $a$ and $b$. Based on the above analysis, the expression of $u_0$ and  the corresponding weighted density can be concluded.
	\end{proof}
	\begin{rem}
		Since the force is degenerate at  $X_0=(0,0)$, the direction of \textbf{\textit{f}}  cannot be determined in this case. Asymptotic directions of the free boundary cannot be determined.
	\end{rem}

	\section{Non-existence of cusp profile}
	In the next proposition, we can show that, for the degenerate weighted density case $M_i(0+)\equiv0$ $(i=1,2,3)$, the cusp phenomenon does not happen.  Without loss of generality, we pick the most complex Case 3 for the proof, as two variables are degenerate for this case. Case 1 and Case 2 are similar, and we omit it here.
	
	\begin{prop} For Type 3 stagnation point $X_0=(0,0)$,
		let $u$ be a weak solution of \eqref{eq1} and suppose that 
		\begin{equation}
			\label{eq73}
			|\nabla u|^2\leq  |x|^{\alpha}|y|^{\beta}
		\end{equation}
		in $B_r(O)$, for some small $r>0$. Then the set of $S^u_{3,\text{cusp}}$ is empty.
	\end{prop}
	\begin{proof}
		Suppose towards a contradiction that $X_0=(0,0)\in S^u_{3,\text{cusp}}$, and
		the blow-up sequence $u_m(X)=r_m^{-\frac{\alpha+\beta+2}{2}}u(r_mX)$ converges strongly in $W_{\text{loc}}^{1,2}(\mathbb{R}^2)$ to a blow-up limit $u_0$. The corresponding blow-up limit of $M_3(0+)=0$ is  $u_0\equiv0$ in $\mathbb{R}^2$, as reasoned above, we have that
		\begin{equation}
			\label{eq14}
			\Delta u_m(B_2)\rightarrow0\,\,\text{as}\,m\rightarrow\infty,
		\end{equation}
		and
		\begin{equation}
			\label{eq95}
			\Delta u_m(B_2)\geq\int_{B_2\cap\p_{red}\{u_m>0\}}\sqrt{|x|^{\alpha}|y|^{\beta}}d\mathcal{H}^1,
		\end{equation}
		where $\Delta u$ is a non-negative Radon measure in $\Omega$ and $ \p_{red} $ is the reduced boundary.
		
		Moreover, we denote that $A=(-1,1)\times(0,1)$, and define a new rectangle $A^*$ as follows,
	$$	A^*= \begin{pmatrix}
			\cos\theta^* & -\sin\theta^* \\
			\sin\theta^* & \cos\theta^* 
		\end{pmatrix}\begin{pmatrix}
		x \\
	y 
	\end{pmatrix},$$
where $(x,y)\in A$.		
		 Based on the maximum principle, there is at least one connected component $V_m$ of $\{u_m>0\}$, such that for any $m$, $V_m$ contains the origin and some points of $ \p A^* $. 
		
		At first, we define 
		$$\zeta_m(x,y)=\max\left\lbrace |x |^{\alpha}|y|^{\beta}|X=(x,y)\in V_m\cap\p A^*\right\rbrace.$$
		As $m\rightarrow\infty$, if $\zeta_m(x,y)$ is bounded away from zero,  a contradiction to \eqref{eq14} and \eqref{eq95} easily follows from the convergence. In fact, if not, when $\zeta_m(x,y)$ converges to zero, we now calculate $\Delta u_m$ in the region $V_m\cap A^*$,
		\begin{equation*}
			\begin{aligned}
				0=&\Delta u_m(V_m\cap A^*)\\
				=&\int_{\p V_m\cap A^*}\nabla u_m\cdot\nu d\mathcal{H}^1+\int_{V_m\cap \p A^*}\nabla u_m\cdot\nu d\mathcal{H}^1,
			\end{aligned}
		\end{equation*}
		together with the boundary condition and \eqref{eq73}, we get that 
		\begin{equation*}
			\begin{aligned}
				0\leq\int_{\p V_m\cap A^*}\left( -\sqrt{|x|^{\alpha}|y|^{\beta}}\right)d\mathcal{H}^1 +\int_{V_m\cap \p A^*}\sqrt{|x|^{\alpha}|y|^{\beta}}d\mathcal{H}^1.
			\end{aligned}
		\end{equation*}
		The last inequality is a contradiction, which concludes the proof.
	\end{proof}
	
	\section{Non-existence of flat profile}
	In this section, we give the detailed analysis for the non-existence of flatness. The main method we used is the frequency formula, this is a classical result, which is first studied by F. Almgren  in \cite{Alm}. However, the gradient function we study here has two variables, the direct frequency formula in \cite{VARVAR2011} cannot be directly applied. It is worth noting that we exclude flat point by perturbation, which is referred in \cite{VARVAR2014}. Accordingly, we have to use remainder term in monotonicity formula to perturbate the frequency, more details will be shown in this section.
	Unless otherwise specified, we take $u$ be a weak solution of \eqref{eq1}.
	
	At first, we investigate Subcase 1.1,
	recalling that $$S^u_{1,\text{flat}}:=\left\lbrace X_0=(x_0,0)\in S^u\Big|M_1(0+)=(-x_0)^{\alpha}\int_{B_1} (y_-)^{\beta}dX\right\rbrace$$
	as the set of flat stagnation points in Subcase 1.1.
	
	\begin{rem}
		We have to point out that the set $S^u_{1,\text{flat}}$ is closed due to the fact that the function $ M_1(0+)$ is upper semicontinuous about $X$.
	\end{rem}
	\begin{lemma}
		Let $X_0$ be a point of the closed set $S^u_{1,\text{flat}}$,  and for some $\widetilde{r_0}\in(0,\delta)$ sufficiently small, we define for $r\in(0,\widetilde{r_0})$
		$$D(r)=r\dfrac{\int_{B_{r}(X_0)}|\nabla u|^2dX}{\int_{\p B_{r}(X_0)}u^2d\mathcal{H}^1},$$
		and 
		\begin{equation}
			\label{eq72}
			V(r)=r\dfrac{\int_{B_{r}(X_0)} \left[ \left( -x_0\right) ^{\alpha}\left( y_-\right) ^{\beta}-(-x)^{\alpha}(-y)^{\beta}\chi_{\{u>0\}}\right]  dX+r^{2+\beta}\int_{0}^{r}h_1(t)dt}{\int_{\p B_{r}(X_0)}u^2d\mathcal{H}^1}.
		\end{equation}
		And we define the frequency function
		$$H(r)=D(r)-V(r).$$
		Then some properties of $V(r)$ and the frequency function $H(r)$ hold as follows.
		
		(i)  The frequency function satisfies $H(r) \geq \frac{\beta}{2}+1$ for all $r\in(0,\widetilde{r_0})$;
		
		(ii) the function $\frac{1}{r}V^2(r)$ is integrable  with respect to $r$ when  $r\in (0,\widetilde{r_0})$;
		
		(iii) the limit of the frequency function $\lim\limits_{r\rightarrow0+}H(r)$ exists, denoted as  $H(0+)$. Obviously, we know that $H(0+)\geq\frac{\beta}{2}+1$.
	\end{lemma}
	\begin{proof}
		We first write $H(r)$ as  
		\begin{equation}
			\label{eq71}
			\begin{aligned}
				H(r)=\dfrac{I_{*}(r)-(-x_0)^{\alpha}\int_{B_1}(y_-)^{\beta}dX}{J_1(r)},
			\end{aligned}
		\end{equation}
		where
		$$I_{*}(r)=I_1(r)-\int_{0}^{r}h_1(t)dt.$$
		
		In order to get the conclusion (i), it is equivalent to the following inequality
		\begin{equation*}
			I_{*}(r)-\left( \dfrac{\beta}{2}+1\right) J_1(r)-r^{-2-\beta}(-x_0)^{\alpha}\int_{B_{r}(X_0)}(y_-)^{\beta}dX\geq0,
		\end{equation*}
		which follows from the direct calculation,
		\begin{equation}
			\label{eq70}
			\begin{aligned}
				&I_{*}(r)-\left( \dfrac{\beta}{2}+1\right) J_1(r)-r^{-2-\beta}(-x_0)^{\alpha}\int_{B_{r}(X_0)}(y_-)^{\beta}dX\\
				=&M_1(r)-M_1(0+)-\int_{0}^{r}h_1(t)dt\\
				=&\int_{0}^{r}\dfrac{dM_1(t)}{dt}dt-\int_{0}^{r}h_1(t)dt\\
				\geq&0.
			\end{aligned}
		\end{equation}
		The above inequality holds since monotonicity formula \eqref{eq2} shows us that the third term of \eqref{eq70} is a perfect square term. Hence, we obtain the statement (i).

		Next together with \eqref{eq76}, we calculate the derivative of $H(r)$ with \eqref{eq71},
		\begin{equation*}
			\begin{aligned}
				\dfrac{	dH(r)}{dr}
				=&\dfrac{I'_1(r)-h_1(r)}{J_1(r)}-\left( D(r)-V(r)\right)\cdot\dfrac{J'_1(r)}{J_1(r)}\\
				=&\dfrac{2r^{-2-\beta}\int_{\p B_{r}(X_0)}(\nabla u\cdot \nu)^2d\mathcal{H}^{1}-(\beta+2)r^{-3-\beta}\int_{\p B_{r}(X_0)}u\nabla u\cdot \nu d\mathcal{H}^{1}}{r^{-3-\beta}\int_{\p B_{r}(X_0)}u^2d\mathcal{H}^1}\\
				&-\left( D(r)-V(r)\right) \dfrac{2r^{-3-\beta}\int_{\p B_{r}(X_0)}u\nabla u\cdot\nu d\mathcal{H}^{1}-(\beta+2)r^{-4-\beta}\int_{\p B_{r}(X_0)}u^{2}d\mathcal{H}^{1}}{r^{-3-\beta}\int_{\p B_{r}(X_0)}u^{2}d\mathcal{H}^{1}}.
			\end{aligned}
		\end{equation*}
		By rearranging the above equation, it follows that
		\begin{equation}
			\begin{aligned}
				\label{eq47}
				\dfrac{dH(r)}{dr}
				=\dfrac{2}{r}\Bigg[& \dfrac{\int_{\p B_{r}(X_0)}\left( r(\nabla u\cdot\nu)-D(r)u\right) ^2d\mathcal{H}^1}{\int_{\p B_{r}(X_0)}u^2d\mathcal{H}^1}
				+V^2(r)+V(r)\left( H(r)-\left( \dfrac{\beta}{2}+1\right) \right) \Bigg]\\
				=\dfrac{2}{r}\Bigg[& \dfrac{\int_{\p B_{r}(X_0)}\left( r(\nabla u\cdot\nu)-H(r)u\right)^2d\mathcal{H}^1}{\int_{\p B_{r}(X_0)}u^2d\mathcal{H}^1}+V(r)\left( H(r)-\left( \dfrac{\beta}{2}+1\right) \right) \Bigg].
			\end{aligned}
		\end{equation}
		Moreover, $V(r)$ can be rewritten as
		\begin{equation*}
			\begin{aligned}
				V(r)
				=\dfrac{V_1(r)+V_2(r)}{\int_{\p B_{r}(X_0)}u^2d\mathcal{H}^1},
			\end{aligned}
		\end{equation*}
		where
		\begin{equation}
			\label{eq96}
			V_1(r)=	 \left( -x_0\right) ^{\alpha}r\int_{B_{r}(X_0)}\left( y_-\right) ^{\beta}(1-\chi_{\{u>0\}})dX,
		\end{equation}
		and 
		\begin{equation}
			\label{eq97}
			V_2(r)=	r\int_{B_{r}(X_0)} \left[ \left( -x_0\right) ^{\alpha}-(-x)^{\alpha}\right] (-y)^{\beta}\chi_{\{u>0\}}dX+r^{3+\beta}\int_{0}^{r}h_1(t)dt.
		\end{equation}
	 As $r<\widetilde{r_0}$, it follows from $V_1(r)\geq0$ that $V(r)\geq0$.
		
		In other words, equation \eqref{eq47} also implies that
		$\frac{2}{r}V^2(r)\leq H'(r),$ the conclusion (ii) and (iii) can be derived together with the fact that $H(r)$ is bounded below as $r\rightarrow0+$, which concludes the proof.
	\end{proof}

	The frequency formula allows us to pass to blow-up limit.
	\begin{prop}
		Let $X_0\in S^u_{1,\text{flat}}$ and  suppose that the weak solution satisfies \eqref{eq69}. Then
		
		(i) $\lim\limits_{r\rightarrow0+}V(r)=0$ and $\lim\limits_{r\rightarrow0+}D(r)=H(0+)$;
		
		(ii) for any infinitesimal sequence $r_m$, where  $m\rightarrow\infty$, we define $v_m$ as 
		\begin{equation}
			\label{eq16}
			v_m(X):=\dfrac{u(X_0+r_mX)}{\sqrt{r_m^{-1}\int_{\p B_{r_m}(X_0)}u^2d\mathcal{H}^1}},
		\end{equation} 
		the sequence $v_m$ 	is bounded in $W^{1,2}(B_1)$;
		
		(iii) suppose that the sequence $v_m$ defined in \eqref{eq16} converges weakly in $W^{1,2}(B_1)$ to a blow-up limit $v_0$, then the function $v_0$ is continuous and is a homogeneous function of $H(0+)$ degree in $B_1$, and satisfies $v_0\geq0$ in $B_1$, $v_0\equiv0$ in $B_1^+$ where $B_1^+=B_1\cap\{X=(x,y)|y\geq0\}$, and $\int_{\p B_1}v^2_0d\mathcal{H}^1=1$.
	\end{prop}
	\begin{proof}
		First, we claim for any sequence $r_m\rightarrow0+$ and  every $0<\rho<\sigma<1$, the sequence $v_m$ defined in \eqref{eq16} satisfies, 
		\begin{equation}
			\begin{aligned}
				\label{eq17}
				\int_{B_{\sigma}\backslash B_{\rho}}|X|^{-4-\beta}\left( \nabla v_m\cdot X-H(0+)v_m\right) ^2dX\,\rightarrow0\,\text{ as }m\rightarrow\infty.
			\end{aligned}
		\end{equation} 
		
		Indeed, for any such $\rho$ and $\sigma$, and every $m$ such that $r_m<\delta$. As $m\rightarrow\infty$, it follows from \eqref{eq47} that,
		\begin{equation}
			\label{eq12}
			\begin{aligned}
				\int_{\rho}^{\sigma}\dfrac{2}{s}\left[ \dfrac{\int_{\p B_{s}}\left( s\nabla v_m\cdot\nu-H(sr_m)v_m\right)^2d\mathcal{H}^1}{\int_{\p B_{s}}v^2_md\mathcal{H}^1}\right]ds
				\leq H(r_m\sigma)-H(r_m\rho)\rightarrow0.
			\end{aligned}
		\end{equation}
		Firstly, we want to show that
		\begin{equation}
			\label{eq48}
			J_1(r)=r^{-3-\beta}\int_{\p B_{r}(X_0)}u^2(X)d\mathcal{H}^1\,\,\,\,\text{is non-decreasing.}
		\end{equation} 
		Together with the monotonicity formula and \eqref{eq70}, we acquire
		\begin{equation*}
			\begin{aligned}
				\dfrac{	dJ_1(r)}{dr}
				=&2r^{-1}\bigg( I_1(r)-\int_{0}^{r}h_1(t)dt-\left( \dfrac{\beta}{2}+1\right) J_1(r)-(-x_0)^{\alpha}\int_{B_{1}}(y_-)^{\beta}dX\\
				&-r^{-2-\beta}	\int_{B_{r}(X_0)}\left( -x\right) ^{\alpha}\left( -y\right) ^{\beta}\chi_{\{u>0\}} dX+\int_{0}^{r}h_1(t)dt+(-x_0)^{\alpha}\int_{B_{1}}(y_-)^{\beta}dX\bigg) \\
				\geq&2r^{-3-\beta}\bigg(	\left( -x_0\right) ^{\alpha}\int_{B_{r}(X_0)}\left( y_-\right) ^{\beta}(1-\chi_{\{u>0\}}) dX\\
				&+r^{2+\beta}\int_{0}^{r}h_1(t)dt+\int_{B_{r}(X_0)} \left[ \left( -x_0\right) ^{\alpha}-(-x)^{\alpha}\right] (-y)^{\beta}\chi_{\{u>0\}}dX\bigg).
			\end{aligned}
		\end{equation*}
		It follows from the similar reason as $V(r)\geq0$ in Lemma 4.2 that we have $J'_1(r)\geq0$.

		Now note that, for every $s\in(\rho,\sigma)\subset(0,1)$ and $m$ as before. Recalling the definition of $v_m$, we calculate the integration of $v^2_m$ on $\p B_{s}$, we have
		\begin{equation*}
			\begin{aligned}
				\int_{\p B_{s}}v^2_md\mathcal{H}^1	&=\dfrac{\int_{\p B_{r_ms}(X_0)}u^2(X)d\mathcal{H}^1}{\int_{\p B_{r_m}(X_0)}u^2(X)d\mathcal{H}^1}\\
				&=\dfrac{(r_ms)^{-3-\beta}\int_{\p B_{r_ms}(X_0)}u^2(X)d\mathcal{H}^1}{r_m^{-3-\beta}\int_{\p B_{r_m}(X_0)}u^2(X)d\mathcal{H}^1}s^{3+\beta}.
			\end{aligned}
		\end{equation*}
		Since $J_1(r)$ is non-decreasing, it follows that
		\begin{equation}
			\label{56}
			\int_{\p B_{s}}v^2_md\mathcal{H}^1\leq s^{3+\beta}.
		\end{equation}
		Combining with \eqref{eq12}, as $m\rightarrow\infty$, we have
		\begin{equation}
			\label{eq13}
			\begin{aligned}
				&\int_{\rho}^{\sigma}\dfrac{2}{s}s^{-3-\beta} \int_{\p B_{s}}\left( s\nabla v_m\cdot\nu-H(sr_m)v_m\right)^2d\mathcal{H}^1ds\rightarrow0.
			\end{aligned}
		\end{equation}
		Thus \eqref{eq17} follows from the fact \eqref{eq13}.

		Next, we use the claim \eqref{eq17} to give the proof of the results in this proposition.
		
		For the result (i), we argue by contradiction. Suppose that $V(s_m)\nrightarrow0$, where $s_m$ is a vanishing sequence of positive constants. It follows from the fact that $\frac{1}{r}V^2(r)$ is integrable with respect to $r\in[s_m,2s_m]$, hence,  that the minimum of $V(r)$ converges to zero as $m\rightarrow\infty$.
		
		Without loss of generality, we may assume that $r_m$ is a sequence such that $r_m\in[s_m,2s_m]$ and $V(r_m)\rightarrow0 $ as $m\rightarrow\infty$.
		In view of the existence of the frequency function $H(0+)$, $D(r_m)$ is bounded as $r_m\rightarrow0+$.	We now calculate  $D(r_m)$ by scaling,
		\begin{equation}
			\label{eq30}
			\begin{aligned}
				D(r_m)=\int_{B_1}|\nabla v_m|^2dX.
			\end{aligned}
		\end{equation}
		Thus (ii) holds and there exists  $v_0\in W^{1,2}(B_1)$ as the weak limit of $v_m$ up to  a subsequence. Together with the trace theorem $W^{1,2}(B_1) \hookrightarrow L^2(\p B_1)$,  we have 
		\begin{center}
			$v_m\rightarrow v_0$ strongly in $L^2(\p B_1)$.
		\end{center}
		The conclusion (iii) now follows by a direct calculation of  $\int_{\p B_1}|v_m|^2dX=1$ and  \eqref{eq17}. It remains to show that (i) holds. Substituting $s_m$ into the definition of $V(r)$ for $r$, 
		\begin{equation}
			\label{eq18}
			\begin{aligned}
				&V(s_m)\\	=&\dfrac{s_m^{-2-\beta}\left[	\int_{B_{s_m}(X_0)} \left( -x_0\right) ^{\alpha}\left( y_-\right) ^{\beta}(1-\chi_{\{u>0\}})+\left[ \left( -x_0\right) ^{\alpha}-(-x)^{\alpha}\right] (-y)^{\beta}\chi_{\{u>0\}}dX +s_m^{2+\beta}\int_{0}^{s_m}h_1(t)dt \right] }{s_m^{-3-\beta}\int_{\p B_{s_m}(X_0)}u^2d\mathcal{H}^1}\\
				\leq&\dfrac{s_m^{-2-\beta}\int_{\p B_{r_m}(X_0)}u^2d\mathcal{H}^1}{\left( \frac{r_m}{2}\right) ^{-2-\beta}\int_{\p B_{\frac{r_m}{2}}(X_0)}u^2d\mathcal{H}^1}V(r_m).
			\end{aligned}
		\end{equation}	
		Due to the definition of $v_m$ in  \eqref{eq16}, we can substitute
		$$	\int_{\p B_{\frac{r_m}{2}}(X_0)}u^2d\mathcal{H}^1$$
		with
		$$\int_{\p B_{\frac{1}{2}}}v_m^2d\mathcal{H}^1\int_{\p B_{r_m}(X_0)}u^2d\mathcal{H}^1$$
		in inequality \eqref{eq18}, that is
		\begin{equation}
			\begin{aligned}
				\label{eq19}
				V(s_m)\leq\left( \dfrac{r_m}{2s_m}\right) ^{2+\beta}\dfrac{V(r_m)}{\int_{\p B_{\frac{1}{2}}}v_m^2d\mathcal{H}^1}
				\leq\dfrac{V(r_m)}{\int_{\p B_{\frac{1}{2}}}v_m^2d\mathcal{H}^1}.
			\end{aligned}
		\end{equation}	
		Next, we claim that
		\begin{equation}
			\label{eq83}
			\int_{\p B_{\frac{1}{2}}}v_m^2d\mathcal{H}^1\rightarrow\int_{\p B_{\frac{1}{2}}}v_0^2d\mathcal{H}^1>0.
		\end{equation}
		Suppose that the claim \eqref{eq83} fails, i.e.  $\int_{\p B_{\frac{1}{2}}}v_0^2d\mathcal{H}^1=0$, this means $v_0=0$ holds almost everywhere  on $\p B_{\frac{1}{2}}$. While	the homogeneity of $v_0$ deduces that $$v_0(X)=(2|X|)^{H(0+)}v_0\left( \frac{X}{2|X|}\right),$$
		$v_0=0$ holds almost everywhere on $\p B_1$,
		this is a contradiction to $\int_{\p B_1}v_0^2d\mathcal{H}^1=1$. 
		
		Thus the fact that inequality \eqref{eq19} and $V(r_m)\rightarrow0$ implies $V(s_m)\rightarrow0$, which leads to a contradiction to the fact that $V(s_m)$ is bounded away from $0$. Therefore we have  $V(r)\rightarrow0$ and $D(r)\rightarrow H(0+)$ as $r\rightarrow0+$. We now conclude the proof of Proposition 4.3.
	\end{proof}
	
	Next proposition  allows us to preserve weak solution in the blow-up limit at stagnation points and show that $S^u_{1,\text{flat}}$ is an empty set.
	\begin{prop}
		Let $X_0\in S^u_{1,\text{flat}}$ and $r_m $ be a vanishing sequence of positive real numbers such that the sequence $v_m$ defined in \eqref{eq16} converging weakly in $W^{1,2}(B_1)$ to a limit $v_0$. Then 
		
		(i) there exists a subsequence $v_m$ that converges strongly to $v_0$ in $W^{1,2}_{\text{loc}}(B_1\backslash\{0,0\})$, and $v_0$ is continuous in $B_1$. Moreover, $\Delta v_0$ is a non-negative Radon measure and   $v_0\Delta v_0=0$ holds in the sense of Radon measure in $B_1$;
		
		(ii) 	at each point $X_0$ of the set $S^u_{1,\text{flat}}$ there exists an integer $N(X_0)\geq \frac{\beta}{2}+1$ such that
		$$H(0+)=N(X_0),$$
		and $\text{as}\,\,r\rightarrow0+$,
		\begin{equation}
			\begin{aligned}
				\label{50}
				\dfrac{u(X_0+rX)}{\sqrt{r^{-1}\int_{\p B_{r}(X_0)}u^2d\mathcal{H}^1}}
				\rightarrow\dfrac{\rho^{N(X_0)}\big|\sin(N(X_0)\min\{\max\{\theta,\pi\},2\pi\})\big|}{\left( \int_{\pi}^{2\pi}\sin^2(N(X_0)\theta)d\theta\right) ^{\frac{1}{2}}},
			\end{aligned}
		\end{equation}
		that is 
		\begin{equation*}
			\begin{aligned}
		v_0=\begin{cases}
			0\,\,\,&\theta\leq\pi,\\
		\dfrac{\rho^{N(X_0)}|\sin(N(X_0)\theta)|}{\left( \int_{\pi}^{2\pi}\sin^2(N(X_0)\theta)d\theta\right) ^{\frac{1}{2}}}\,\,\,\,&\pi<\theta<2\pi,\\
		0&\theta\geq2\pi,
				\end{cases}
			\end{aligned}
		\end{equation*}
		where $X=(\rho\cos\theta,\rho\sin\theta)$.
	\end{prop}
	\begin{proof}
		This is a standard result follows directly from Theorem 9.1 and Theorem 10.1 in \cite{VARVAR2011}, therefore we omit the detail of the proof here.
		
		(i) Based on Theorem 1.1 and 3.1 in  \cite{Evans} and
		Theorem 8.17 in
		\cite{Gilbarg}, we can obtain that $v_m\rightarrow v_0$ strongly in  $W^{1,2}_{\text{loc}}(B_1\backslash\{0,0\})$. As a consequence of the strong convergence, we obtain that $v_0\Delta v_0=0$ in the sense of Radon measure in $B_1$.
		
		(ii) Due to the homogeneity of $v_0$ and $\Delta v_0=0$ in $ \{v_0>0\} $, 
		\begin{equation*}
			\begin{aligned}
				v_0(\rho,\theta)=\begin{cases}
					\dfrac{\rho^{N(X_0)}|\sin(N(X_0)\theta)|}{\sqrt{\int_{\pi}^{2\pi}\sin^2\left( N(X_0)\theta\right) d\theta}}\,\,&\theta\in(\theta_1,\theta_2),\\
					0&\theta\notin(\theta_1,\theta_2).
				\end{cases}
			\end{aligned}
		\end{equation*}
		In addition, for $v_0\equiv0$ in $\mathbb{R}^2\cap\{X=(x,y)|y\geq0\}$,  $\theta_1=\pi$ and  $\theta_2=
		2\pi$.
	\end{proof}

	\begin{cor}
		$S^u_{1,\text{flat}}$ is empty.
	\end{cor} 
	\begin{proof}
		Suppose towards a contradiction that $S^u_{1,\text{flat}}\neq	\varnothing $, there exists a point $X_0\in S^u_{1,\text{flat}}$. By Proposition 4.4 (ii), as $r\rightarrow0+$,	\begin{center}
			$\dfrac{u(X_0+rX)}{\sqrt{r^{-1}\int_{\p B_{r}(X_0)}u^2d\mathcal{H}^1}}\rightarrow\dfrac{\rho^{N(X_0)}|\sin(N(X_0)\min\{\max\{\theta,\pi\},2\pi\})|}{\left( \int_{\pi}^{2\pi}\sin^2(N(X_0)\theta)d\theta\right) ^{\frac{1}{2}}}$
		\end{center}
		strongly in $W^{1,2}_{\text{loc}}(B_1\backslash\{0,0\})$ and weakly in $W^{1,2}(B_1)$, where $X=(\rho\cos\theta,\rho\sin\theta).$
		
		In fact, $X_0\in S^u_{1,\text{flat}}$, then for any ball $\tilde{B} \subset\subset B_1\cap\{X=(x,y)|y<0\}$, $v_{m}>0$ in $\tilde{B}$ for sufficiently small $r_m$. On the other  hand, $v_0$ is harmonic in $\tilde{B}$, contradicting \eqref{50} in view of $N(X_0)\geq\frac{\beta}{2}+1$. Therefore $S^u_{1,\text{flat}}$ is empty.
	\end{proof}

	The proof of	Subcase 2.1 can be completed by the same  method as employed in Subcase 1.1. However, the proof for Case 3 is slightly different and worth mentioning here. Similarly, we can get the following conclusion. Reviewing the previous definition, $$S^u_{3,\text{flat}}:=\left\lbrace X=(x,y)\in S^u\Big|M_3(0+)=\int_{B_1} |x|^{\alpha}|y|^{\beta}dX\right\rbrace$$
	as the set of flat stagnation points in Case 3.
	
	\begin{prop}
		Consider the functions as follows
		$$\overline{D}(r)=r\dfrac{\int_{B_{r}}|\nabla u|^2dX}{\int_{\p B_{r}}u^2d\mathcal{H}^1},$$$$\overline{V}(r)=r\dfrac{\int_{B_{r}} |x|^{\alpha}|y|^{\beta}\left( 1-\chi_{\{u>0\}}\right) dX}{\int_{\p B_{r}}u^2d\mathcal{H}^1},$$
		and	the “frequency” function is
		$$\overline{H}(r)=\overline{D}(r)-\overline{V}(r).$$
		
		Then 	for some $\widetilde{r_0}\in(0,\delta)$ sufficiently small, the following conclusions hold
		
		(i)  $\lim\limits_{r\rightarrow0+}\overline{V}(r)=0$ and $\lim\limits_{r\rightarrow0+}\overline{D}(r)=\overline{H}(0+)$;
		
		(ii) for any sequence $r_m\rightarrow0+$ as $m\rightarrow\infty$, the sequence
		\begin{equation}
			\label{eq52}
			\overline{v}_m(x):=\dfrac{u(r_mX)}{\sqrt{r_m^{-1}\int_{\p B_{r_m}}u^2d\mathcal{H}^1}}
		\end{equation} 
		is bounded in $W^{1,2}(B_1)$;
		
		(iii) suppose that the sequence $\overline{v}_m$ in \eqref{eq52} converges weakly in $W^{1,2}(B_1)$ to a blow-up limit $\overline{v}_0$, the function $\overline{v}_0$ is continuous and a homogeneous function of  $\overline{H}(0+)$ degree in $B_1$, and satisfies $\overline{v}_0\geq0$ in $B_1$, $\overline{v}_0\equiv0$ in $B_1\cap\{X|X\cdot\nu^*\geq0\}$, and $\int_{\p B_1}\overline{v}^2_0d\mathcal{H}^1=1$;
		
		(iv)  $\overline{v}_m$ converges to $\overline{v}_0$ strongly in  $W^{1,2}_{\text{loc}}(B_1\backslash\{0,0\})$,  and $\Delta \overline{v}_0$ is a non-negative Radon measure satisfying $\overline{v}_0\Delta \overline{v}_0=0$ in the sense of Radon measure in $B_1$;
		
		(v)	 there exists an integer $\overline{N}(O)\geq \frac{\alpha+\beta}{2}+1$ such that
		$\overline{H}(0+)=\overline{N}(O),$ 
		and $\text{as}\,\,r\rightarrow0+$,
		\begin{equation}
			\begin{aligned}
				\label{51}
				\dfrac{u(rX)}{\sqrt{r^{-1}\int_{\p B_{r}}u^2d\mathcal{H}^1}}
				\rightarrow\dfrac{\rho^{\overline{N}(O)}|\sin(\overline{N}(O)\min\{\max\{\theta,\theta^*-\frac{\pi}{2}\},\theta^*+\frac{\pi}{2}\}+\overline{N}(O)(\frac{\pi}{2}+\theta^*))|}{\left( \int_{\theta^*-\frac{\pi}{2}}^{\theta^*+\frac{\pi}{2}}\sin^2(\overline{N}(O)\theta+\overline{N}(O)(\frac{\pi}{2}-\theta^*))d\theta\right) ^{\frac{1}{2}}}
			\end{aligned}
		\end{equation}
		strongly in $W^{1,2}_{\text{loc}}(B_1\backslash(0,0))$ and weakly in $W^{1,2}(B_1)$, where $X=(\rho\cos\theta,\rho\sin\theta)$.
	\end{prop}
	\begin{cor}
		$S^u_{3,\text{flat}}$ is empty.
	\end{cor}
	\begin{proof}
		Suppose towards a contradiction that 	$S^u_{3,\text{flat}}\neq	\varnothing $, $X_0=(0,0)\in S^u_{3,\text{flat}}$. By Proposition 4.6(v), as $r\rightarrow0+$,	\begin{center}
			$\dfrac{u(rX)}{\sqrt{r^{-1}\int_{\p B_{r}}u^2d\mathcal{H}^1}}\rightarrow\dfrac{\rho^{\overline{N}(O)}|\sin(\overline{N}(O)\min\{\max\{\theta,\theta^*-\frac{\pi}{2}\},\theta^*+\frac{\pi}{2}\}+\overline{N}(O)(\frac{\pi}{2}+\theta^*))|}{\left( \int_{\theta^*-\frac{\pi}{2}}^{\theta^*+\frac{\pi}{2}}\sin^2(\overline{N}(O)\theta+\overline{N}(O)(\frac{\pi}{2}-\theta^*))d\theta\right) ^{\frac{1}{2}}}$
		\end{center}
		strongly in $W^{1,2}_{\text{loc}}(B_1\backslash\{0,0\})$ and weakly in $W^{1,2}(B_1)$, where $X=(\rho\cos\theta,\rho\sin\theta).$
		
		In fact, for any ball $\tilde{B} \subset\subset B_1\cap\{X=(x,y)|X\cdot(\cos\theta^*,\sin\theta^*)<0\}$, $\overline{v}_m>0$ in $\tilde{B}$ for sufficiently small $r_m$, see Fig. 18. On the other  hand, $\overline{v}_0$ is harmonic in $\tilde{B}$, contradicting \eqref{51} in view of $\overline{N}(O)\geq\frac{\alpha+\beta}{2}+1$. Hence $S^u_{3,\text{flat}}$ is empty.	
		\begin{figure}[h]
			\includegraphics[width=50mm]{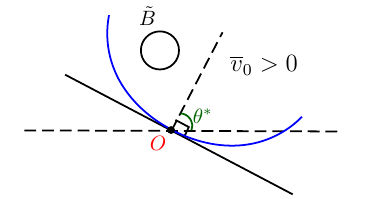}	
			\caption{The location of  $\tilde{B}$}
		\end{figure}
	\end{proof}
	
	\section{Asymptotic directions of free boundary}
	Under the assumption that the free boundary is locally an injective curve, we now derive its asymptotic behavior as it approaches a stagnation point. Without loss of generality, taking Subcase 1.1 as an example to prove the conclusion, the proof of other cases  follows from Subcase 1.1, so we omit here.
	\begin{prop}
		Let $u$ be a weak solution of  free boundary problem \eqref{eq1} satisfying \eqref{eq69}. Suppose in addition that $\p\{u>0\}$ is in a neighborhood of $X_0$ a continuous injective curve $\sigma=(\sigma_1(t),\sigma_2(t)):I\rightarrow\mathbb{R}^2$ such that $\sigma(0)=X_0$, where $I$ is an interval of $ \mathbb{R}^2 $ containing the origin. When  $$M_1(0+)=(-x_0)^{\alpha}\int_{B_1\cap\left\lbrace -\frac{\beta+4}{2(\beta+2)}\pi <\theta<-\frac{\beta}{2(\beta+2)}\pi\right\rbrace } (y_-)^{\beta}dX,$$
		then  $\sigma_1(t)\neq x_0$ in $(-s,s)\backslash\{0\}$ and, depending on the parametrization, either
		\begin{center}
			$\lim\limits_{t\rightarrow0+}\dfrac{\sigma_2(t)}{\sigma_1(t)-x_0}=\tan\left( -\frac{\beta+4}{2(\beta+2)}\pi\right) $ and $\lim\limits_{t\rightarrow0-}\dfrac{\sigma_2(t)}{\sigma_1(t)-x_0}=\tan\left( -\frac{\beta}{2(\beta+2)}\pi\right) ,$
		\end{center}
		or
		\begin{center}
			$\lim\limits_{t\rightarrow0+}\dfrac{\sigma_2(t)}{\sigma_1(t)-x_0}=\tan\left(-\frac{\beta}{2(\beta+2)}\pi \right) $ and $\lim\limits_{t\rightarrow0-}\dfrac{\sigma_2(t)}{\sigma_1(t)-x_0}=\tan\left( -\frac{\beta+4}{2(\beta+2)}\pi\right) .$
		\end{center}
	\end{prop}
	\begin{proof}
		We define arg$x$ as the complex argument of $x$ and the sets 
		$$A_{\pm}:=\left\lbrace \tilde{\theta}\in \left[- \pi,0\right] : \text{there is }\, t_m\rightarrow0\pm\,\text{such that arg}(\sigma(t_m)-\sigma(0))\rightarrow\tilde{\theta},\,\text{as}\,m\rightarrow\infty\right\rbrace .$$
		
		We claim that 
		$$A_{\pm}\subset\left\lbrace -\pi,-\frac{\beta+4}{2(\beta+2)}\pi ,-\frac{\beta}{2(\beta+2)}\pi,0\right\rbrace .$$
		
		Suppose towards a contradiction that a sequence $\{t_m\}$, such that as $m\rightarrow\infty$, $ t_m\rightarrow0,$ and \begin{center}
			arg$\sigma(t_m)-$arg$\sigma(0)\rightarrow\tilde{\theta}\in(A_+\cup A_-)\backslash\left\lbrace -\pi,-\frac{\beta+4}{2(\beta+2)}\pi ,-\frac{\beta}{2(\beta+2)}\pi,0\right\rbrace$.
		\end{center} 
		Let $r_m:=|\sigma(t_m)-\sigma(0)|$ and 
		$$u_m(X)=\dfrac{u(X_0+r_mX)}{r_m^{\frac{\beta+2}{2}}}.$$
		
		We denote $\tilde{X}=(\cos\tilde{\theta},\sin\tilde{\theta})$, obviously $|\tilde{X}|=1$.  And setting lines 
		\begin{equation*}
			l_1: y=\tan\left(-\frac{\beta}{2(\beta+2)}\pi\right)x,
		\end{equation*}
		and 
		\begin{equation*}
			l_2: y=\tan\left( -\frac{\beta+4}{2(\beta+2)}\pi\right)x.
		\end{equation*}
		For each $\rho>0$, we take a ball $B_{\rho}(\tilde{X})$, which does not intersect $l_1$, $l_2$ or $x$-axis. 
		
		Moreover, we suppose that $X_m=\frac{\sigma(t_m)-\sigma(0)}{r_m}$, then
		$$u_m(X_m):=\dfrac{u(X_0+\sigma(t_m)-\sigma(0))}{r_m^{\frac{\beta+2}{2}}}=\dfrac{u(\sigma(t_m))}{r_m^{\frac{\beta+2}{2}}},$$
		since $\sigma(t_m)\in\p\{u>0\}$, $u_m(X_m)=0$. That is to say, 
		$$X_m\in\p\{u_m>0\}.$$
	Clearly, $|X_m|=1$ and \begin{center}
			arg$X_m=$arg$\dfrac{\sigma(t_m)-\sigma(0)}{r_m}\rightarrow\tilde{\theta}$,
		\end{center} therefore $X_m\rightarrow \tilde{X}$. Hence  $B_{\rho/2}(X_m)\subset B_{\rho}(\tilde{X})$ holds when $m$ is large enough.
		
		We already get that 
		$$\Delta u_m\left( B_{\rho}(\tilde{X})\right) =\int_{B_{\rho}\cap\p\{u_m>0\}}\sqrt{(-x_0)^{\alpha}(_-y)^{\beta}}d\mathcal{H}^1\geq C'\rho, $$
		where $C'$ is a positive constant. However,  $u_m\rightarrow u_0$ strongly in $W^{1,2}_{\operatorname{loc}}(\mathbb{R}^2)$, we deduce that 
		$$\Delta u_m(B_{\rho}(\tilde{X}))\rightarrow\Delta u_0(B_{\rho}(\tilde{X}))=0,\,\,\text{as}\,\,m\rightarrow\infty,$$
		in the sense of measure, which contradicts with the fact that $\Delta u_m(B_{\rho}(\tilde{X}))>0$.
		
		If there exists a sequence	$\{t_m\}$ such that	$\sigma_1(t_m)=\sigma_1(0)$ as $t_m\rightarrow0$, then
		$$\arg\left( \sigma(t_m)-\sigma(0)\right) \rightarrow-\frac{\pi}{2},$$
		it follows from that $-\frac{\pi}{2}\in 
		A_{\pm}$, which contradicts with the claim. Therefore, we can deduce that $\sigma_1(t)\neq\sigma_1(0)$ for sufficiently small $t\neq0$. 
		
		Moreover, a continuity argument yields that both $A_+$ and $A_-$ are connected sets, that is, $A_+$ and $A_-$ contain only one element respectively. Consequently, we define the elements in $A_+$ and $A_-$ as 
		\begin{center}
			$a_+:=\lim\limits_{t\rightarrow0+}\operatorname{arg}\sigma(t)$
			and
			$a_-:=\lim\limits_{t\rightarrow0-}\operatorname{arg}\sigma(t)$
		\end{center}
		respectively.
		
		For $\rho>0$, we take $$\hat{B}_1=B_{\rho} \left( \cos\left(-\frac{\beta+4}{2(\beta+2)}\pi \right),\sin\left( -\frac{\beta+4}{2(\beta+2)}\pi\right) \right)$$ and    $$\hat{B}_2=B_{\rho} \left( \cos\left(-\frac{\beta}{2(\beta+2)}\pi \right),\sin\left( -\frac{\beta}{2(\beta+2)}\pi\right) \right) .$$ It is easy to see that $\Delta u_0(\hat{B}_i)>0,i=1,2$, which implies that $$\{a_+,a_-\}=\left\lbrace -\frac{\beta}{2(\beta+2)}\pi,-\frac{\beta+4}{2(\beta+2)}\pi\right\rbrace .$$ Thus we complete the proof.				 
	\end{proof}

	\section*{Appendix}
	\appendix
	The main results of this paper are presented in Table 1.
	
	\begin{table}[h]
		\caption{Conclusion}\label{table1}
		\centering
		\begin{tabular}{|m{1.7em}<{\centering}|m{4.3em}<{\centering}|m{4.6em}<{\centering}|m{4em}<{\centering}|m{2.5em}<{\centering}|m{16em}<{\centering}|m{2.5em}<{\centering}|}
			\hline
			Type&Stagnation point
			&Location of stagnation point &Force direction & Angle &  Weighted density& Profile
			\\
			\hline
			\multirow{4}{*}{1} & 	\multirow{4}{*}{$X_0$$=(x_0,0)$} & $x_0<0$ & $\theta_0=\frac{3\pi}{2}$& \multirow{4}{*}{ $\frac{2\pi}{\beta+2}$} & $(-x_0)^{\alpha}\int_{B_1\cap\left\lbrace -\frac{\beta+4}{2(\beta+2)}\pi <\theta<-\frac{\beta}{2(\beta+2)}\pi\right\rbrace } (y_-)^{\beta}dX$ & Fig. 6(a)\\
			\cline{3-4}\cline{6-7}
			& & $x_0>0$ & $\theta_0=\frac{\pi}{2}$& &$(x_0)^{\alpha}\int_{B_1\cap\left\lbrace \frac{\beta}{2(\beta+2)}\pi <\theta<\frac{\beta+4}{2(\beta+2)}\pi\right\rbrace } (y_+)^{\beta}dX$ & Fig. 6(b)\\
			\cline{3-4}\cline{6-7}
			& & $x_0<0$ & $\theta_0=\frac{\pi}{2}$ & & $(-x_0)^{\alpha}\int_{B_1\cap\left\lbrace \frac{\beta}{2(\beta+2)}\pi <\theta<\frac{\beta+4}{2(\beta+2)}\pi\right\rbrace } (y_+)^{\beta}dX$ & Fig. 7(a)\\
			\cline{3-4}\cline{6-7}
			& & $x_0>0$ & $\theta_0=\frac{3\pi}{2}$ & & $(x_0)^{\alpha}\int_{B_1\cap\left\lbrace -\frac{\beta+4}{2(\beta+2)}\pi <\theta<-\frac{\beta}{2(\beta+2)}\pi\right\rbrace } (y_-)^{\beta}dX$& Fig. 7(b)\\
			\hline
			\multirow{4}{*}{2} & 	\multirow{4}{*}{$X_0$$=(0,y_0)$} & $y_0<0$ & $\theta_0=\pi$ & \multirow{4}{*}{ $\frac{2\pi}{\alpha+2}$} & $(-y_0)^{\beta}\int_{B_1\cap\left\lbrace\frac{\alpha+1}{\alpha+2}\pi<\theta<\frac{\alpha+3}{\alpha+2}\pi\right\rbrace } (x_-)^{\alpha}dX$& Fig. 8(a)\\
			\cline{3-4}\cline{6-7}
			& & $y_0>0$ & $\theta_0=0$ & &$(y_0)^{\beta}\int_{B_1\cap\left\lbrace-\frac{\pi}{\alpha+2}<\theta<\frac{\pi}{\alpha+2}\right\rbrace } (x_+)^{\alpha}dX$ & Fig. 8(b)\\
			\cline{3-4}\cline{6-7}
			& & $y_0<0$ & $\theta_0=0$ & &$(-y_0)^{\beta}\int_{B_1\cap\left\lbrace-\frac{\pi}{\alpha+2}<\theta<\frac{\pi}{\alpha+2}\right\rbrace } (x_+)^{\alpha}dX$ & Fig. 9(a)\\
			\cline{3-4}\cline{6-7}
			& & $y_0>0$ & $\theta_0=\pi$ & &$(y_0)^{\beta}\int_{B_1\cap\left\lbrace\frac{\alpha+1}{\alpha+2}\pi<\theta<\frac{\alpha+3}{\alpha+2}\pi\right\rbrace } (x_-)^{\alpha}dX$ & Fig. 9(b)\\
			\hline
			3 & $X_0$$=(0,0)$ & $(0,0)$ & N/A & $\frac{2\pi}{\alpha+\beta+2}$ & $\int_{B_1\cap\left\lbrace \theta_1 <\theta<\theta_2\right\rbrace } |x|^{\alpha}|y|^{\beta}dX$& Fig. 10\\
			\hline
		\end{tabular}
	\end{table}
	
	\section*{Acknowledgments}
The authors would like to thank the anonymous referees for careful reading and for the valuable comments.

\section*{Conflict of interests/Competing Interests}	
The authors declare that they 
have no conflict of interests/competing interests.

\section*{Data availability statement}
My manuscript has no associated data.

\end{document}